\documentclass[a4paper,11pt]{article}
\usepackage[top=3.0cm, bottom=3.0cm, inner=3.0cm, outer=3.0cm, includefoot]{geometry}

\usepackage[utf8]{inputenc}
\usepackage[T1]{fontenc}
\usepackage{authblk}
\usepackage{amsfonts}
\usepackage{verbatim}
\usepackage{multirow}
\usepackage{multicol}

\usepackage{amssymb}
\usepackage{amsmath}
\usepackage{mathtools,bbm}
\usepackage{graphicx,tikz}
\usepackage{amsthm}
\usepackage{caption,subcaption}

\usepackage{color}
\usepackage{enumerate}

\usepackage{cancel}

\newcommand{\supp}{\operatorname{supp}}

\newcommand{\eChar}{\begin{enumerate}[(i)]}
\newcommand{\eCharR}{\begin{enumerate}[(a)]}
\newcommand{\eBr}{\begin{enumerate}[(1)]}

\newcommand{\Abstract}

\title
{
Symmetric Matrices, Signed Graphs, and \\ Nodal Domain Theorems
}

\author{Chuanyuan Ge*}\author{Shiping Liu}

\affil{School of Mathematical Sciences,University of Science and Technology of China, Hefei 230026, China, \\
*Correspondence to be sent to: e-mail: gechuanyuan@mail.ustc.edu.cn}

\date{}

\theoremstyle{plain}
\newtheorem{lemma}{Lemma}[section]
\newtheorem{theorem}[lemma]{Theorem}
\newtheorem{proposition}[lemma]{Proposition}
\newtheorem{corollary}[lemma]{Corollary}

\theoremstyle{definition}

\newtheorem{definition}[lemma]{Definition}
\newtheorem{remark}[lemma]{Remark}

\newtheorem{example}[lemma]{Example}

\numberwithin{equation}{section}

\begin{document}

\maketitle

\pagestyle{plain}

\begin{abstract}
In 2001, Davies, Gladwell, Leydold, and Stadler proved discrete nodal domain theorems for eigenfunctions of generalized Laplacians, i.e., symmetric matrices with non-positive off-diagonal entries. In this paper, we establish nodal domain theorems for arbitrary symmetric matrices by exploring the induced signed graph structure. Our concepts of nodal domains for any function on a signed graph are switching invariant. When the induced signed graph is balanced, our definitions and upper bound estimates  reduce to existing results for generalized Laplacians. Our approach provides a more conceptual understanding of Fiedler's results on eigenfunctions of acyclic matrices. This new viewpoint leads to lower bound estimates for the number of strong nodal domains which improves previous results of Berkolaiko and Xu-Yau. We also prove a new type of lower bound estimates by a duality argument.
\end{abstract}

\section{Introduction}
Courant's nodal domain theorem is a basic result in spectral theory with wide applications.
The theorem, proved in 1920s \cite{Courant23,CH}, states that the nodal lines of the $k$-th eigenfunction $f_k$ of a self-adjoint second order elliptic differential operator can not divide the domain $D$ into more than $k$ different subdomains. Here the nodal lines are refered to the set of zeros (nodes) of eigenfunctions and the subdomains are now known as the nodal domains. Courant's theorem can be considered as a natural generalization of Sturm’s oscillation theorem for second order ODEs that the zeros of the $k$-th eigenfunction of a vibrating string divide the string into exactly $k$ subintervals. There are abundant extensions of Courant's theorem to non-linear operators like $p$-Laplacians, and to Riemannian manifolds and more general settings with less regularity, including discrete settings, see, e.g., \cite{CSZ17,CFG00,Cheng76,JMZ21,JZ2106,KS20,TH18}.

The study of discrete nodal domain theorems on graphs dates back to the work of Gantmacher and Krein \cite{GK50}, which contains a discrete analogue of Sturm's theorem for strings. Many of Fiedler's results in 1970s \cite{Fiedler73,Fiedler75,Fiedler752} can be interpreted as discrete nodal domain estimates.
Important progresses in this aspect can be found in the works of Powers \cite{Powers88}, Roth \cite{Roth89}, Friedman \cite{Friedman93},  Colin de Verdi\`{e}re \cite{CdV93}, van der Holst \cite{vdH95,vdH96}, Duval and Reiner \cite{DR99}, etc. The discrete nodal domain theorems for \emph{generalized Laplacians}, i.e., symmetric matrices with non-positive off-diagonal entries, were eventually established by Davies, Gladwell, Leydold, and Stadler \cite{DGLS01} in 2001. We refer to \cite[Section 2]{DGLS01} for a detailed historical review. There are many further advances in this topic, see, e.g., \cite{B03,BLS05,BHLPS04,GZ02,LLMY10,Lovasz21} and the book \cite{BLS07}.

Nodal domain theorems for graphs have striking difference from their $\mathbb{R}^d$ analogue. One of the key steps in establishing the discrete nodal domain theorems is to come up with the correct concept of nodal domains on graphs: We need distinguish strong and weak nodal domains \cite[Definitions 1 and 2]{DGLS01} (see also Definition \ref{def:DGLS} below) since the zeros are discrete. The nodal domain theorems in \cite{DGLS01} states that the number of strong nodal domains of the $k$-th eigenfunction of a generalized Laplacian is no greater than $k+r-1$, where $r$ is the multiplicity of the corresponding eigenvalue; When the corresponding graph is connected, the number of weak nodal domains of the $k$-th eigenfunction is no greater than $k$. The strong nodal domain estimates can not be improved due to the example of star graphs given in Friedman \cite{Friedman93}. The weak nodal domain estimates has been correctly stated in the work of Colin de Verdi\`{e}re \cite{CdV93} and Friedman \cite{Friedman93} and a complete proof was given in \cite{DGLS01}.  While there are no nontrivial lower bound for the number of nodal domains of $D\subset\mathbb{R}^d$, $d\geq 2$,  Berkolaiko \cite{Berkolaiko08} established a non-trivial lower bound estimate for the number of strong nodal domains on graphs, which were strengthened later by Xu and Yau \cite{XY12}. This can be considered as strong extensions of the result of Sturm and its discrete counterparts: When the graph is not far from being a tree and the eigenfunction has only few zeros, the number of corresponding strong nodal domains is very close to the upper bound $k+r-1$.
%

It is natural to ask for discrete nodal domain theorems of general symmetric matrices. This includes important cases of signed graph Laplacian \cite{AL20}  and related Schr\"odinger operators \cite{BGKLM20} on finite grpahs or on subgraphs with Dirichlet boundary condition. One motivation comes from the finite element method (FEM) approximation of differential equations. Gladwell and Zhu \cite{GZ02} studied the nodal domain theorem for an approximate FEM solution to the Helmholtz equation corresponding to some refined or crude mesh. Recall that the FEM reduces the Helmholtz equation to a form $M f=0$, where $M$ is a symmetric matrix. The signs of the off-diagonal entries of $M$ depend on the characteristic of the mesh: For the case of a triangular mesh in $\mathbb{R}^2$, $M$ is a generalized Laplacian if all the triangles are acute-angled; Some off-diagonal entries of $M$ might be positive if some triangles are obtuse-angled. Gladwell and Zhu \cite{GZ02} mentioned that it is easy to construct counterexamples of meshes with some obtuse-angled triangles for which the discrete nodal domain theorems in \cite{DGLS01} fails. Mohammadian \cite{Mohammadian16} has proved the strong nodal domain theorem for any symmetric matrices. However, his definition of "weak" nodal domains is different from \cite[Definition 2]{DGLS01}: For a star graph on $n\geq 3$ vertices, any eigenfunction to the second Laplacian eigenvalue has two weak nodal domains while it only has one "weak" nodal domain in the sense of Mohammadian (see Remark \ref{rmk:Mohammadian}).  For the case of signless graph Laplacians, i.e., symmetric matrices with all off-diagonal entries non-negative, the so-called nonzero nodal domains of a function (maximal connected induced subgraphs on non-zeros of the function) have been studied in \cite{JMZ21,JZ2106}. The signs of the function values do not play a role in this approach.

In this paper, we introduce proper concepts of nodal domains for arbitrary symmetric matrices and establish various upper and lower bound estimates which unifies the approach of Davies et al. \cite{DGLS01} for generalized Laplacians and that of Fiedler \cite{Fiedler75} for acyclic matrices.
We do this via the induced signed graph of a symmetric matrix.
A signed graph $\Gamma=(G,\sigma)$ is a graph $G=(V,E)$ whose edges are labelled by a signature $\sigma: E\to \{+1,-1\}$. By an induced signed graph $\Gamma$ of an $n\times n$ symmetric matrix $M=(M_{ij})$, we mean a graph $G$ with vertices $\{x_1,\ldots, x_n\}$, where $\{x_i,x_j\}$ is an edge if and only if $i\neq j$ and $M_{ij}\neq 0$, and the sign of each edge is $\sigma_{x_ix_j}=-M_{ij}/|M_{ij}|$.
In the definitions \cite[Definitions 1 and 2]{DGLS01} (see Definition \ref{def:DGLS}), the strong and weak nodal domains of a function $f$ are decided by the sign of $f$ at each vertex. In our concepts (Definition \ref{def:nodaldomain}), the signature of edges also plays an important role. We introduce strong and weak nodal domain walks (see Definitions \ref{def:strong} and \ref{def:weak} below) which further induce two kinds of equivalent relations on the set of non-zeros of a function $f$. Building upon the corresponding equivalent classes, we define strong and weak nodal domains of a function $f$ (Definition \ref{def:nodaldomain}). When the symmetric matrix is a generalized Laplacian, that is, when the induced signed graph has all-positive signature, our definition coincides with that of \cite[Definitions 1 and 2]{DGLS01}.

Signed graphs and the fundamental ideas of balance and switching has led to systematic and deep understandings for various parts of graph theory, e.g., for matroid theory \cite{Zaslavsky82} and for Cheeger constants and related eigenvalue estimates \cite{AL20,LMP19}. The sign of a cycle is defined to be the product of the signs of all its edges. A signed graph is called balanced if the sign of every cycle is positive. If the induced signed graph $\Gamma$ of a symmetric matrix $M$ is balanced, then there exists a diagonal matrix $D(\tau)$ of a function $\tau: V\to \{+1,-1\}$, i.e., $D_{ii}=\tau(x_i)$ for any $i$, such that $$M^\tau:=D(\tau)^{-1}MD(\tau)=D(\tau)MD(\tau)$$ is a generalized Laplacian. Notice that the symmetric matrix $M$ and the generalized Laplacian matrix $M^\tau$ share the same set of eigenvalues and
\[Mf_k=\lambda_kf_k \,\,\text{if and only if}\,\,M^\tau \tau f_k=\lambda_k\tau f_k.\]
It turns out that the strong and weak nodal domains in our sense (Definition \ref{def:nodaldomain}) of $f_k$ as an eigenfunction of $M$ coincide with the strong and weak nodal domains in the sense of \cite[Definitions 1 and 2]{DGLS01} (see Definition \ref{def:DGLS}) of $\tau f_k$ as an eigenfunction of the generalized Laplacian $M^\tau$. In general, our definition of strong and weak nodal domains are switching invariant (Theorem \ref{Thm:switch invariant}).

With the new definition of nodal domains on signed graphs, we extend the theorems of Davies et al. \cite{DGLS01} to arbitrary symmetric matrices (Theorems \ref{Thm:upper bound}). Particularly, a discrete unique continuation theorem using our concept of weak nodal domains holds for eigenfunctions of any symmetric matrix (Lemma \ref{lemma:unique}). Building upon this fact, we prove the number of weak nodal domains of the $k$-th eigenfunction of any symmetric matrix with a connected induced signed graph is no greater than $k$.  We also show the number of strong nodal domains of the $k$-th eigenfunction of any symmetric matrix with \emph{minimal support} (Definition \ref{def:minisupp}) is no greater than $k$.

Our approach using signed graphs leads to more conceptual understanding of existing results. For example, Roth \cite{Roth89} proved that for the generalized Laplacian on a connected bipartite graph, the largest eigenvalue $\lambda_n$ is simple and the numbers of strong and weak nodal domains of the corresponding  eigenfunction $f_n$ are both equal to $n$.  Since any bipartite graph with an all-positive signature is antibalacned, Roth's result becomes a direct consequence of the Perron-Frobenius theorem (see Theorem \ref{Thm:n}). Fiedler \cite{Fiedler75} studied eigenfunctions of acyclic matrices and found close relations between the signs of eigenfunction values and the positions of the corresponding eigenvalues. We reformulate Fiedler's results as discrete nodal domain estimates for symmetric matrices whose induced signed graphs have no cycle.
Notice that any signed graph with no cycle is balanced.

Inspired by this reformulation of Fiedler's work, we obtain an interesting multiplicity formula (Corollary \ref{cor:multiplicitytilde}), and lower bound estimates for the number of strong nodal domains, which improves the results of Berkolaiko \cite{Berkolaiko08} and Xu and Yau \cite{XY12} even in the balanced case. A particular subset of zeros, which we call \emph{Fiedler zero set} (see Definition \ref{def:Fiedler zero set}), plays an important role in those results.

 In our approach,  a duality argument via considering the quantity
 \[\mathfrak{S}(f)+\overline{\mathfrak{S}}(f)\]
is quite useful, where $\mathfrak{S}(f)$ is the number of strong nodal domains of the function $f$ on a signed graph $\Gamma=(G,\sigma)$ and $\overline{\mathfrak{S}}(f)$ is that of $f$ on its negation $-\Gamma=(G,-\sigma)$. Using this argument, we prove a new type of lower bound estimates for the number of strong nodal domains involving properties of leaves, i.e., vertices with degree $1$ (Theorem \ref{Thm:lower bound 2}).


The paper is structured as follows. In Section \ref{section:pre}, we collect preliminaries on symmetric matrices and signed graphs, particularly on the concepts of balance, antibalance, and switching of signed graphs. In Section \ref{section:concepts}, we present the concepts of strong and weak nodal domains on singed graphs, and discuss their basic properties. In Section \ref{section:upper}, we prove discrete nodal domain theorems for arbitrary symmetric matrices. In Section \ref{Section:Fiedler}, we reformulate the main theorem of Fiedler \cite{Fiedler75} on the eigenfunctions of acyclic matrices as discrete nodal domain estimates on signed graphs with no cycle, and derive an interesting multiplicity formula. Finally, we show two lower bound estimates for the number of strong nodal domains of any symmetric matrix in Section \ref{Section:lowerbd}.

\section{Preliminaries}\label{section:pre}
We first give the following definition which we will used often in this paper.
\begin{definition}[walk and path]
	A walk in a graph $G=(V,E)$ is a  sequence  of vertices $\{v_i\}_{i=1}^k$ for $k\geq 2$, such that $\{v_j,v_{j+1}\}\in E$ for $1\leq j \leq k-1$. A path is a walk such that all vertices $v_i$ are distinct. 
\end{definition}
Let $M$ be an $n\times n$ symmetric matrix. We list its eigenvalues with multiplicity as follows:
\[\lambda_1 \leq \lambda_2 \leq \cdot\cdot\cdot\cdot\leq \lambda_n.\]
Recall the following mini-max principle.
\begin{lemma}\label{Lemma:minmax}
Let $P_k$ and $P_k^{\perp}$ be the sets of subspaces of $\mathbb{R}^n$ with dimension at least $k$ and with codimension at most $k$, respectively. Then
	\begin{equation}
		\lambda_k=
		\min \limits_{P\in P_k} \max \limits_{0\neq g\in P}\frac{\langle g,Mg \rangle}{\langle g,g \rangle}=\max \limits_{P\in P_{k-1}^{\perp}}\min \limits_{0\neq g\in P\in P_k}\frac{\langle g,Mg \rangle}{\langle g,g \rangle},		
	\end{equation}
where $\langle \cdot,\cdot\rangle$ stands for the Euclidean inner product of $\mathbb{R}^n$.
\end{lemma}

\begin{definition}\label{def:minisupp}
We say an eigenfunction $f$ of $M$ has \emph{minimal support} if for each eigenfunction $g$ corresponding to the same eigenvalue as $f$ with $\supp(g)\subseteq \supp(f)$ one has $\supp(g)= \supp(f)$. Here we use the notion $\supp(f):=\{i\in \{1,\ldots,n\}: f(i)\neq 0\}.$
\end{definition}

\begin{lemma}\label{lemma:2.3}
For any $n\times n$ symmetric matrix $M$, there exists a basis $\{f_1,\ldots,f_n\}$ consisting of eigenfunctions with minimal support.
\end{lemma}
\begin{proof}
	We show for any eigenspace $E$ there exisits a basis consisting of functions with minimal support. First choose a basis $\{f_{i}:=(y_{i1}, y_{i2}\ldots, y_{in})^T:\,i=1,\ldots r \}$ of $E$ where $r=\dim E$. Consider the following $r\times n$ matrix
	\[M_0=\left(
	\begin{array}{cccc}
	y_{11} & y_{12} & \cdots &  y_{1n} \\
	y_{21} & y_{22} &\cdots &  y_{2n} \\
		\vdots & \vdots& \ddots & \vdots\\
	y_{r1}& y_{r2} &\cdots & y_{rn} \\
	\end{array}
	\right).\]
Without loss of generality, we can assume that $M_0$ can be transformed via elementary row transformation into the following form
	\[\left(
\begin{array}{ccccccc}
     1  & 0 & \cdots &0&z_{1,r+1}& \ldots&z_{1n} \\
	0 & 1 &\cdots & 0&z_{2,r+1}& \ldots&z_{2n} \\
	\vdots & \vdots& \ddots & \vdots&\vdots&\vdots&\vdots\\
	0& 0 &\cdots &1&z_{r,r+1}&\ldots& z_{rn} \\
\end{array}
\right).\]
Set $\eta_i:=(\delta_{i1}, \ldots., \delta_{ir}, z_{i,r+1}, \ldots, z_{in})^T,\, i=1,\ldots,\,r$. Then $\{\eta_i,\, i=1,\ldots,\,r\}$ is a basis of the eigenspace $E$. Next we show each $\eta_i$ has minimal support by contradiction.
Suppose that we have a function $f=(x_1,\ldots x_n)^T\in E$ with $\supp(f)\varsubsetneq\supp(\eta_i)$. This implies that $x_j=0,$ for $j=1,\ldots ,r, j\neq i$. If $x_i=0$, then $f\not\in E$. If, otherwise, $x_i\neq 0$, then $f$ is a multiple of $\eta_i$. Contradiction.
\end{proof}

A signed graph $\Gamma=(G,\sigma)$ is a graph $G=(V,E)$ with a signature $\sigma: E\to \{+1,-1\}$, where $V$ is the vertex set and $E$ is the edge set.  We say two vertices $x,y\in V$ are connected by an edge if $\{x,y\}\in E$ and write $x\sim y$. We denote by $d_x:=\sum_{y:y\sim x}1$ the vertex degree of $x\in V$. For each edge $\{x,y\}\in E$, we write $\sigma_{xy}=\sigma(\{x,y\})$ for short. In this article, we concern the induced signed graphs $\Gamma=(G,\sigma)$  of symmetric matrices, for which the graphs $G$ are always undirected, finite, and simple.

\begin{definition}
	Given an $n\times n$ symmetric matrix $M=(M_{ij})$, we define the \emph{induced signed graph} $\Gamma=(G,\sigma)$ of $M$ as follows: The underlying graph $G=(V,E)$ is given by \[V:=\{ x_i \}_{i=1}^n\,\,\text{ and }\,\,E:=\{ \{x_i,x_j\}: M_{ij}\neq 0 \,\,\text{and}\,\, i\neq j \},\] and the signature $\sigma: E\to \{+1,-1\}$ is defined via
\[\sigma_{x_ix_j}:=-\frac{M_{ij}}{|M_{ij}|}=\left\{
                                            \begin{array}{ll}
                                              +1, & \hbox{if $M_{ij}<0$;} \\
                                              -1, & \hbox{if $M_{ij}>0$,}
                                            \end{array}
                                          \right.
\]
for each edge $\{x_i,x_j\}\in E$.
\end{definition}

By the above definition, the induced signed graph of any generalized Laplacian always has an all-positive signature.

\begin{definition}
 Given a signed graph $\Gamma=(G,\sigma)$. We say a symmetric matrix $M$ is \emph{compatible} with $\Gamma$ if the induced signed graph of $M$ coincides with $\Gamma$.
\end{definition}

In a signed graph $\Gamma=(G,\sigma)$, the sign of a cycle or a path  in $G$ is defined as the product of the signs of edges in it. The following concepts of balance and antibalance are introduced by Harary \cite{Harary55,Harary57}.

\begin{definition}[Harary]
	A signed graph $\Gamma=(G,\sigma)$ is called balanced if the sign of every cycle in G is positive. It is called antibalanced if the sign of every odd cycle is negative and the sign of every even cycle is positive.
\end{definition}
Notice that a signed graph $\Gamma=(G,\sigma)$ is antibalanced if and only if its negation $-\Gamma:=(G,-\sigma)$ is balanced.

\begin{definition}[Switching]
	A function $\tau:V\to \{+1,-1\}$ is called a switching function. Switching the signature of $\Gamma=(G,\sigma)$ by $\tau$ refers to the operation of changing $\sigma$ to be $\sigma^\tau $ where
	\[\sigma^{\tau}_{xy}:=\tau(x)\sigma_{xy}\tau(y)\]
	for any $\{x,y\}\in E$.
\end{definition}
Switching the signature $\sigma$ by $\tau$ is simply reversing the signs of all edges connecting $S:=\{x\in V: \tau(x)=+1\}$ and its complement while keeping the signs of other edges unchanged.
\begin{definition}
	Let $G=(V,E)$ be a graph. Two signatures $\sigma: E\to \{+1.-1\}$ and $\sigma':E\to \{+1,-1\}$ are called to be switching equivalent if there exists a switching function $\tau$ such that $\sigma'=\sigma^\tau$.
\end{definition}
Notice that the sign of a cycle is invariant under switching operations. The following characterization lemma using switching is due to Zaslavsky \cite[Corollary 3.3]{Zaslavsky82}.
\begin{lemma}[Zaslavsky's switching lemma]
	A signed graph $\Gamma=(G,\sigma)$ is balanced if and only if $\sigma$ is switching equivalent to the all-positive signature, and it is antibalanced if and only if $\sigma$ is switching equivalent to the all-negative signature.
\end{lemma}

We can always switch the signs of all edges in a spanning tree of a signed graph to be positive \cite[Lemma 3.1]{Zaslavsky82}.
\begin{lemma}\label{lemma:2.8}
	Let $\Gamma=(G,\sigma)$ be a connected signed graph. Then for any spanning tree $T$ of $G$, there exists a switching function $\tau$ such that $\sigma$ can be switched by $\tau$ to be positive on each edge of $T$.
\end{lemma}

\begin{proof}
 Select one vertex $x_0\in V$. The switching function constructed below fulfills the requirement:
\[ \tau=\left\{
          \begin{array}{ll}
            +1, & \hbox{if $x=x_0$;} \\
            \text{the sign of the unique path connecting } x \text{ and } x_0 \,\text{in }\, T, & \hbox{if $x\neq x_0$.}
          \end{array}
        \right.
 \]

\end{proof}

Let $\Gamma=(G,\sigma)$ be the induced signed graph of a symmetric matrix $M$. Consider a switching function $\tau: V\to \{+1,-1\} $. Then the symmetric matrix \[M^{\tau}:=D(\tau)MD(\tau),\] where $D(\tau)$ is the diagonal matrix of $\tau$, shares the same spectrum with $M$. The induced signed graph of $M^\tau$ is the switched graph $\Gamma^\tau=(G, \sigma^\tau)$. Indeed, a function $f_k$ is an eigenfunction of $M$ such that $Mf_k=\lambda_kf_k$ if and only if $\tau f_k$ is an eigenfunction of $M^\tau$ such that $M^\tau (\tau f_k)=\lambda_k (\tau f_k)$.


\section{Strong and weak nodal domains on signed graphs}\label{section:concepts}
In this section, we present the concepts of strong and weak nodal domains on signed graphs in detail.

\begin{definition}[Strong nodal domain walks]\label{def:strong}
		Let $\Gamma=(G,\sigma)$ be a signed graph where $G=(V,E)$ and $f: V\to \mathbb{R}$ be a function. A walk $\{ x_k \}_{k=1}^n$, $n\geq 2$ is called a \emph{strong nodal domain walk} of $f$ (an S-walk for short) if $f(x_k)\sigma_{x_kx_{k+1}} f(x_{k+1})>0$ for each $k=1,2,\ldots,n-1$.
\end{definition}
Notice that $f(x_k)\neq 0$ for any vertex $x_k$ in an S-walk of $f$. In contrast, we allow zeros in the following concept of weak nodal domain walks.
\begin{definition}[Weak nodal domain walks]\label{def:weak}
		Let $\Gamma=(G,\sigma)$ be a signed graph where $G=(V,E)$ and $f: V\to \mathbb{R}$ be a function. A walk $\{ x_k \}_{k=1}^n$, $n\geq 2$ is called a \emph{weak nodal domain walk} of $f$ (a W-walk for short) if for any two consecutive non-zeros $x_i$ and $x_j$ of $f$, i.e., $f(x_i)\neq 0$, $f(x_j)\neq 0$, and $f(x_{\ell})=0$ for any $i<\ell<j$, it holds that
\[f(x_{i}) \sigma_{x_{i}x_{i+1}}\cdots\sigma_{x_{j-1}x_{j}}f(x_j)>0.\]
\end{definition}
We remark that every walk containing at most $1$ non-zeros of $f$ is a $W$-walk.

Using the above two types of walks, we introduce the following two equivalence relations on the set of non-zeros of a function $f$.
\begin{definition}
 Let $\Gamma=(G,\sigma)$ be a signed graph where $G=(V,E)$ and $f: V\to \mathbb{R}$ be a function. Let $\Omega=\{ v\in V:f(v)\neq 0 \}$ be the set of non-zeros of $f$.
\begin{itemize}
  \item [(i)] We define a relation $R_S$ on $\Omega$ as follows: For any $x,y\in \Omega$, $(x,y)\in R_S$ if and only if $x=y$ or there exists an S-walk connecting $x$ and $y$.
  \item [(ii)] We define a relation $R_W$ on $\Omega$ as follows: For any $x,y\in \Omega$, $(x,y)\in R_W$ if and only if $x=y$ or there exists an W-walk connecting $x$ and $y$.
\end{itemize}
\end{definition}
It is direct to check that both $R_S$ and $R_W$ are equivalence relations.

\begin{definition}[Strong and weak nodal domains]\label{def:nodaldomain}
Let $\Gamma=(G,\sigma)$ be a signed graph where $G=(V,E)$ and $f: V\to \mathbb{R}$ be a function. Let $\Omega=\{ v\in V:f(v)\neq 0 \}$ be the set of non-zeros of $f$.
\begin{itemize}
  \item [(i)] We denote by $\{S_i\}_{i=1}^p$ the equivalence classes of the relation $R_S$ on $\Omega$. We call the induced subgraph of each $S_i$ a \emph{strong nodal domain} of the function $f$.
 We denote the number $p$ of strong nodal domains of $f$ by $\mathfrak{S}(f)$.
  \item [(ii)]We denote by $\{W_i\}_{i=1}^q$ the equivalence classes of the relation $R_W$ on $\Omega$. We call the induced subgraph of each set \[W_i^0:=W_i\cup\{x\in V: \text{there exists a W-walk from $x$ to some vertex in $W_i$}\}\] a \emph{weak nodal domain} of the function $f$. We denote the number $q$ of weak nodal domains of $f$ by $\mathfrak{W}(f)$.
\end{itemize}
\end{definition}
Notice that $W_i^0$ is obtained from $W_i$ by absorbing the zeros around it.
\begin{remark}
 \begin{itemize}
   \item [(i)] From the definition, we have $\mathfrak{S}(f)=\mathfrak{W}(f)=0$ if $f$ is identically zero.
   \item [(ii)] We observe that both strong and weak nodal domains of a function are connected.
   \item [(iii)] For any two weak nodal domains $D_i$ and $D_j$ of $f$, if $x\in D_i\cap D_j$, then $x$ must be a zero, i.e., $f(x)=0$.
 \end{itemize}
\end{remark}

For the number of strong nodal domains $\mathfrak{S}(f)$ of a function $f$, we have the following observation.
\begin{lemma}\label{lemma:strong}
Let $\Gamma=(G,\sigma)$ be a signed graph with $G=(V,E)$.
 For a function $f:V\to \mathbb{R}$ and the set of non-zeros $\Omega=\{x\in V: f(x)\neq 0\}$, consider the subgraph $S=(\Omega, E(S))$  where
 \[E(S):=\{\{x,y\}\in E: f(x)\sigma_{xy}f(y)>0\}.\]
  Let $T_S=(\Omega, E(T_S))$ be a spanning forest of $S$. Then we have
\[\mathfrak{S}(f)=|V|-z-|E(T_S)|,\]
where $z=|V\setminus \Omega|$ is the number of zeros of $f$.
\end{lemma}
\begin{proof}
By definition, the number $\mathfrak{S}(f)$ is the number of connected components of the subgraph $S$ or that of its spanning forest $T_S$. Therefore, we have
\[\mathfrak{S}(f)=|\Omega|-|E(T_S)|=|V|-z-|E(T_S)|.\]
\end{proof}

Next, we illustrate our concept by an example.
\begin{example}\label{ex:fish}
	We consider the signed graph $\Gamma=(G,\sigma)$ given in Figure \ref{fig:product_scheme} and the symmetric matrix
 \[M=\left(
\begin{array}{cccccc}
	 0 & -1 & 0 & -1 & 0 & 0 \\
	-1 &  0 &-1 &  1 & 0 & 0 \\
     0 & -1 & 0 & -1 &-1 &-1 \\
	-1 &  1 &-1 &  0 & 0 & 0 \\
	 0 &  0 &-1 &  0 & 0 & 1 \\
	 0 &  0 &-1 &  0 & 1 & 0 \\
\end{array}
\right),\]
which is compatible with $\Gamma$.
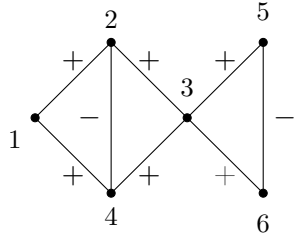
\begin{figure}[!htp]
	\centering
	\tikzset{vertex/.style={circle, draw, fill=black!20, inner sep=0pt, minimum width=3pt}}
	\begin{tikzpicture}[scale=1.0]
		\draw (1,1) -- (0,0) node[midway, above, black]{$+$}
		-- (1,-1) node[midway, below, black]{$+$};
		\draw (2,0) -- (1,-1) node[midway,below, black]{$+$}
		-- (1,1) node[midway,left, black]{$-$}
		-- (2,0) node[midway, above, black]{$+$} ;
		\draw (2,0) -- (3,-1)node[midway,below,black]{+} -- (3,1)node[midway,right,black]{$-$} -- (2,0)node[midway,above,black]{$+$};
		
		\node at (0,0) [vertex, label={[label distance=0mm]225: \small $1$}, fill=black] {};
		\node at (1,-1) [vertex, label={[label distance=0mm]270: \small $4$} ,fill=black] {};
		\node at (1,1) [vertex, label={[label distance=0mm]90: \small $2$} ,fill=black] {};
		\node at (2,0) [vertex, label={[label distance=1mm]90: \small $3$} ,fill=black] {};
		\node at (3,1) [vertex, label={[label distance=1mm]90: \small $5$} ,fill=black] {};
		\node at (3,-1) [vertex, label={[label distance=1mm]270: \small $6$} ,fill=black] {};
		
	\end{tikzpicture}
	\caption{$\Gamma=(G,\sigma)$.}
	
	\label{fig:product_scheme}
\end{figure}
 By numerical computation, we obtain the eigenvalues of $M$ listed below:
\[\lambda_{1}\approx-1.84\leq \lambda_{2}=\lambda_{3}=-1\leq \lambda_{4}\approx-0.51\leq\lambda_{5}\approx1.51\leq\lambda_{6}\approx2.84.\]
The following are a system of corresponding eigenfunctions:
\begin{align*}
f_1&\approx(1.76,1.62,2.84,1.62,1,1)^T,\\
f_2&=(0,-1,0,1,0,0)^T,\\
f_3&=(0,0,0,0,-1,1)^T,\\
f_4&\approx(-2.44,-0.62,1.51,-0.62,1,1)^T,\\
f_5&\approx(0.82,-0.62,-0.51,-0.62,1,1)^T,\\
f_6&\approx(-1.14,1.62,-1.84,1.62,1,1)^T.
\end{align*}
We list the strong and weak nodal domains of each eigenfunction in Table \ref{t}. Notice that we only provide vertex subsets. The strong and weak nodal domains are the induced subgraphs of those vertex subsets in Table \ref{t}.
\begin{table}[!ht]
	\begin{center}
		\begin{tabular}{c|c|c} 
			\textbf{Eigenfunction} & \textbf{Strong nadal domain} & \textbf{Weak nodal domain}\\
		
			\hline
			$f_1$ & $\{1,2,3,4,5,6\}$& $\{1,2,3,4,5,6\}$\\
			$f_2$ & $\{2,4\}$& $\{1,2,3,4,5,6\}$\\
			$f_3$ & $\{5,6\}$ & $\{1,2,3,4,5,6\}$\\
			$f_4$&$\{1,2,4\},\{3,5,6\}$&$\{1,2,4\},\{3,5,6\}$\\
			$f_5$&$\{1\},\{2,3,4\},\{5\},\{6\}$&$\{1\},\{2,3,4\},\{5\},\{6\}$\\
			$f_6$&$\{1\},\{2\},\{3\},\{4\},\{5\},\{6\}$&$\{1\},\{2\},\{3\},\{4\},\{5\},\{6\}$\\
		\end{tabular}
\caption{Strong and weak nodal domains}\label{t}
	\end{center}
\end{table}
We also illustrate the eigenfunction $f_6$ in Figure \ref{fig:2}. There we see none of the edges is an $S$-walk.

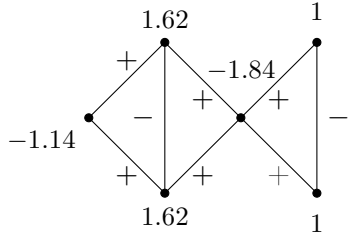
\begin{figure}[!htp]
	\centering
	\tikzset{vertex/.style={circle, draw, fill=black!20, inner sep=0pt, minimum width=3pt}}
	\begin{tikzpicture}[scale=1.0]
		\draw (1,1) -- (0,0) node[midway, above, black]{$+$}
		-- (1,-1) node[midway, below, black]{$+$};
		\draw (2,0) -- (1,-1) node[midway,below, black]{$+$}
		-- (1,1) node[midway,left, black]{$-$}
		-- (2,0) node[midway, below, black]{$+$} ;
		\draw (2,0) -- (3,-1)node[midway,below,black]{+} -- (3,1)node[midway,right,black]{$-$} -- (2,0)node[midway,below,black]{$+$};
		
		\node at (0,0) [vertex, label={[label distance=0mm]225: \small $-1.14$}, fill=black] {};
		\node at (1,-1) [vertex, label={[label distance=0mm]270: \small $1.62$} ,fill=black] {};
		\node at (1,1) [vertex, label={[label distance=0mm]90: \small $1.62$} ,fill=black] {};
		\node at (2,0) [vertex, label={[label distance=3mm]90: \small $-1.84$} ,fill=black] {};
		\node at (3,1) [vertex, label={[label distance=1mm]90: \small $1$} ,fill=black] {};
		\node at (3,-1) [vertex, label={[label distance=1mm]270: \small $1$} ,fill=black] {};
		
	\end{tikzpicture}
	\caption{The function $f_6$ on $\Gamma=(G,\sigma)$.}\label{fig:2}
	
	\label{fig:product_scheme2}
\end{figure}

\end{example}

\subsection{Switching invariance}
Let us recall the definition of strong and weak nodal domains of a function on an unsigned graph from \cite[Definitions 1 and 2]{DGLS01}.
\begin{definition}[\cite{DGLS01}]\label{def:DGLS}
Let $G=(V,E)$ be a graph and $f: V\to \mathbb{R}$ be a function. A positive (negative) strong nodal domain of $f$ is a maximal connected induced subgraph of $G$ on vertices $v\in V$ with $f(v)>0$ ($f(v)<0$). A positive (negative) weak nodal domain of $f$ is a maximal connected induced subgraph of $G$ on vertices $v\in V$ with $f(v)\geq 0$ ($f(v)\leq 0$) that contains at least one nonzero vertex.
\end{definition}
\begin{remark}
Our Definition \ref{def:nodaldomain} of strong and weak nodal domains of a function on a signed graph $\Gamma=(G,\sigma)$ coincides with Definition \ref{def:DGLS} when $\sigma_{xy}=+1$ for any edge $\{x,y\}$.
\end{remark}

\begin{theorem}[Switching invariance]\label{Thm:switch invariant}
Let $f$ be a non-zero function on a signed graph $\Gamma=(G,\sigma)$. Let $\tau: V\to \{+1,-1\}$ be a switching function. Then an induced subgraph of $G$ is a strong (weak) nodal domain of the function $f$ on $\Gamma=(G,\sigma)$ if and only if it is a strong (weak) nodal domain of the function $\tau f$ on $\Gamma^\tau=(G,\sigma^\tau)$.
\end{theorem}
\begin{proof}
The theorem follows directly from the following observation:
For any walk $x_0\sim x_1\cdots\sim x_m,\,m\geq 1$, it holds that
\[f(x_0)\sigma_{x_0x_1}\cdots\sigma_{x_{m-1}x_m}f(x_m)=(\tau f)(x_0)\sigma^\tau_{x_0x_1}\cdots\sigma^\tau_{x_{m-1}x_m}(\tau f)(x_m).\]
\end{proof}

\begin{corollary}
 Let $f$ be a non-zero function on a balanced signed graph $\Gamma=(G,\sigma)$. Let $\tau: V\to \{+1,-1\}$ be the switching function such that $\sigma^\tau_{xy}=+1$ for any edge $\{x,y\}$. Then an induced subgraph of $G$ is a strong (weak) nodal domain of the function $f$ on $\Gamma=(G,\sigma)$ if and only if it is a strong (weak) nodal domain of the function $\tau f$ on the graph $G$ in the sense of Definition \ref{def:DGLS}.
\end{corollary}

\begin{remark}\label{rmk:Mohammadian}
Mohammadian \cite{Mohammadian16} has studied nodal domain theorems for symmetric matrices. Let $M$ be a symmetric matrix and $G=(V,E)$ be the induced graph. Let $f:V\to \mathbb{R}$ be a function. Mohammadian defines the following two subgraphs: One subgraph $\Gamma^{<}_{M,f}(G)$ has vertex set $\Omega:=\{x\in V: f(x)\neq 0\}$ and edge set
\[E^<:=\{\{x,y\}\in E: f(x)M_{xy}f(y)<0\}.\]
The other subgraph $\Gamma^{\leq}_{M,f}(G)$ has vertex set $V$ and edge set
\[E^\leq:=\{\{x,y\}\in E: f(x)M_{xy}f(y)\leq 0\}.\]
If the graph $G$ is connected, Mohammadian proves for the $k$-th eigenfunction $f_k$ of $M$ that
\[c(\Gamma^{\leq}_{M,f_k}(G))\leq k,\,\,\text{and}\,\,\,c(\Gamma^{<}_{M,f_k}(G))\leq k+r-1,\]
where $c(\cdot)$ is the number of connected components and $r$ is the multiplicity of the $k$-th eigenvalue.

Comparing with our definitions, we have
\[c(\Gamma^{\leq}_{M,f_k}(G))\leq \mathfrak{W}(f_k),\,\,\text{and}\,\,\,c(\Gamma^{<}_{M,f_k}(G))=\mathfrak{S}(f_k).\]
Therefore, Mohammadian has proved that $\mathfrak{S}(f_k)\leq k+r-1$ in our terminology.
Notice that for many cases, e.g., for the second Laplacian eigenfunction of a star graph, the strict inequality \[c(\Gamma^{\leq}_{M,f_k}(G))<\mathfrak{W}(f_k)\] holds.
\end{remark}

\subsection{Antibalance and duality}
When the signed graph $\Gamma=(G,\sigma)$ is balanced, it holds that $\mathfrak{W}(f_1)=\mathfrak{S}(f_1)=1$ for the first eigenfunction $f_1$ of any symmetric matrix compatible with $\Gamma$, by Theorem \ref{Thm:switch invariant} and Perron-Frobenius theorem. Conversely, if $\mathfrak{W}(f_1)=\mathfrak{S}(f_1)=1$, then $\Gamma$ is not necessarily balanced, see Example \ref{ex:fish}. However, for antibalanced case, we have the following result.

\begin{theorem}\label{Thm:n}
   Let $\Gamma=(G,\sigma)$ be a connected signed graph. For any $n\times n$ symmetric matrix $M$ compatible with $\Gamma$, let $\lambda_{n} $ be the largest eigenvalue and $f_n$ be an eigenfunction of $M$ corresponding to $\lambda_{n}$. Then $\Gamma=(G,\sigma)$ is antibalanced if and only if $\mathfrak{W}(f_n)=\mathfrak{S}(f_n)=n$. Moreover, when $\Gamma$ is antibalanced, $\lambda_n$ is simple.
\end{theorem}
\begin{proof}
  We first assume that $\Gamma$ is antibalanced. By Theorem \ref{Thm:switch invariant}, we can further assume that each edge is negative and hence each off-diagonal entry of $M$ is nonnegative. Then we apply Perron-Frobenius theorem to derive that $\lambda_{n}$ is simple and $f_n$ is positive on all vertices. Therefore, each edge of $\Gamma$ is neither a $W$-walk nor an $S$-walk. We have $\mathfrak{W}(f_n)=\mathfrak{S}(f_n)=n$ by definition.

  Next, we assume that $\mathfrak{W}(f_n)=\mathfrak{S}(f_n)=n$. In particular, we know that $f_n$ is nonzero on all vertices. We consider the following switching function
  \[\tau(x)= \frac{f_n(x)}{|f(_nx)|},\,\,\text{for all}\,\,x\in V.\]
  Since $\mathfrak{S}(f_n)=n$, each edge of $\Gamma$ is not an $S$-walk, i.e., $\tau(x)\sigma_{xy}\tau(y)=-1$ for any $\{x,y\}\in E$. That is, we can switch $\sigma$ by $\tau$ such that $\sigma^\tau\equiv -1$. Hence, $\Gamma=(G,\sigma)$ is antibalanced.
\end{proof}
In Example \ref{ex:fish}, we have $\mathfrak{S}(f_6)=\mathfrak{W}(f_6)=6$. By Theorem \ref{Thm:n}, the singed graph in Figure \ref{t} is antibalanced. Of course, we can check the antibalancedness directly by observing that each cycle has a negative sign.

As a consequence of Theorem \ref{Thm:n}, we can derive a result due to Roth \cite{Roth89} directly.
\begin{theorem}\cite{Roth89}\label{Thm:Roth}
	Let $G=(V,E)$ be a connected bipartite graph and $M$ be a generalized Laplacian of $G$. Let $\lambda_{n} $ be the largest eigenvalue and $f_n$ be an eigenfunction of $M$ corresponding to $\lambda_{n}$. Then $\lambda_n$ is simple and we have $\mathfrak{W}(f_n)=\mathfrak{S}(f_n)=n$.
\end{theorem}
\begin{proof}
Any bipartite graph $G$ with an all-positive signature $\sigma$ is antibalanced. Then the facts that $\lambda_n$ is simple and $\mathfrak{W}(f_n)=\mathfrak{S}(f_n)=n$ follow from Theorem \ref{Thm:n}.
\end{proof}

It is a very useful philosophy to consider the strong nodal domains of a function $f$ on a signed graph $\Gamma=(G,\sigma)$ and that of $f$ on its negation $-\Gamma=(G,-\sigma)$. The latter can be considered a dual version of strong nodal domains of $f$. When the graph $G$ is a forest, we have the following identity.
\begin{theorem}\label{Thm:duality}
Let $\Gamma=(G,\sigma)$ be a signed graph where $G=(V,E)$ is a forest with $|V|=n$. For any function $f: V\to \mathbb{R}$,
we have
\[\mathfrak{S}(f)+\overline{\mathfrak{S}}(f)=n+c-2z+e_0.\]
where $\overline{\mathfrak{S}}(f)$ is the number of strong nodal domains of $f$ on the signed graph $-\Gamma=(G,-\sigma)$, $z=|\{x\in V: f(x)=0\}|$, $e_0=|\{\{x,y\}\in E: f(x)=0 \,\,\text{or}\,\,f(y)=0\}|$, and $c$ is the number of connected components of $G$.
\end{theorem}
\begin{proof}
For any edge $\{x,y\}\in E$ of the tree with $f(x)\neq 0$ and $f(y)\neq 0$, exactly one of the following two holds:
\begin{itemize}
  \item [(i)] $\{x,y\}$ is an $S$-walk of $f$ on $\Gamma=(G,\sigma)$;
  \item [(ii)]$\{x,y\}$ is an $S$-walk of $f$ on $-\Gamma=(G,-\sigma)$;
\end{itemize}
Let $p$ and $q$ be the numbers of edges which are $S$-walk of $f$ on $\Gamma$ and on $-\Gamma$, respectively. Then we have $p+q=n-c-e_0$.

By Lemma \ref{lemma:strong}, we obtain $\mathfrak{S}(f)=n-z-p$ and $\overline{\mathfrak{S}}(f)=n-z-q$.
%
Therefore, we have
\[\mathfrak{S}(f)+\overline{\mathfrak{S}}(f)=2n-2z-(p+q)=2n-2z-n+c+e_0=n-2z+c+e_0.\]
This finishes the proof.
\end{proof}
Theorem \ref{Thm:duality} will be employed to calculate eigenvalue multiplicities in Corollary \ref{cor:fiedler} and Corollary \ref{cor:multiplicitytilde} below.

\subsection{Basic properties of weak nodal domains}
We say two domains $D_i$ and $D_j$ are \emph{adjacent}, denoted by $D_i\sim D_j$, if there exist $x\in D_i$ and $y\in D_j$ such that $x\sim y$. By definition, we have the following proposition.
  \begin{proposition}\label{prop:connect}
  	Let $\{D_i\}_{i=1}^q$ be the weak nodal domains of a non-zero function $f$ on a signed graph $\Gamma=(G,\sigma)$. Let $G_D=(V_D, E_D)$ be the graph given by
\[V_D:=\{D_i\}_{i=1}^q,\,\,\text{and}\,\,E_D:=\{\{D_i,D_j\}: D_i\sim D_j\}.\]
If the graph $G$ is connected, so does $G_D$.
  \end{proposition}
\begin{proof}
Let $D$ and $D'$ be any two weak nodal domains. Choose two vertices $x$ and $x'$ such that $x\in D$ and $x'\in D'$. Since the graph $G$ is connected, there exists a walk $x=x_0\sim x_1\sim\cdots\sim x_m=x'$ connecting $x$ and $x'$. Set $i_0:=\max\{i: x_i\in D\}$. Then we have $f(x_{i_0+1})\neq 0$. Therefore, $x_{i_0+1}$ and $x$ belongs to different equivalent classes of the relation $R_W$, i.e., $x_{i_0+1}$ lies in a weak nodal domain $D_1\sim D$. Applying the above argument iteratively, we find a walk $D\sim D_1\sim\cdots\sim D'$ from $D$ to $D'$ in the graph $G_D$. That is, the graph $G_D$ is connected.
\end{proof}
For any zero vertex, we have the following observation.
\begin{proposition}\label{prop:1sphere}
Let $f$ be a non-zero function on a signed graph $\Gamma=(G,\sigma)$. Then for any three weak nodal domain $D_1, D_2, D_3$ of $f$ we have
$D_1\cap D_2\cap D_3=\emptyset$.
\end{proposition}
\begin{proof}
Suppose that $D_1\cap D_2\cap D_3\neq\emptyset$. Let $x\in D_1\cap D_2\cap D_3$. Then we have $f(x)=0$. By definition, we can find for each $i\in \{1,2,3\}$ a walk in $D_i$
\[x^i_0\sim x^i_1\sim \cdots \sim x^i_{p_i}=x\]
such that $f(x^i_0)\neq 0$ and $f(x^i_j)=0$ for any $j\in\{1,\ldots,p_i\}$. Set
\[a_i:=f(x^i_0)\sigma_{x^i_0x^i_1}\cdots \sigma_{x_{p_i-1}^ix}\in \mathbb{R},\,\,i=1,2,3.\]
Since $D_1,D_2,D_3$ are different from each other, we obtain
\[a_1a_2<0, \,\,a_2a_3<0,\,\,\text{and}\,\,a_3a_1<0,\]
which is a contradiction.
\end{proof}
\begin{corollary}\label{cor:3.8}
 Let $f$ be a non-zero function on a signed graph $\Gamma=(G,\sigma)$. Let $x$ be a vertex lying in two weak nodal domains $D$ and $D'$. Then the set
\begin{align*}
B(x):=\{y\in V: &\,\,\text{there exist  $m\in \mathbb{Z}$ and a walk}\,\,x=x_0\sim x_1\sim\cdots\sim x_m=y\\
&\,\,\text{such that} \,\,f(x_i)=0, i=0,1,\ldots,m-1.\}
\end{align*}
is contained in $D\cup D'$. In particular, we have $S_1(x)\subset D\cup D'$ where $S_1(x)=\{y\in V: y\sim x\}$.
\end{corollary}
\begin{proof}
 Let $y\in B(x)$. If $f(y)$=0, then we have $y\in D\cup D'$ by definition. In the case of $f(y)\neq 0$, we suppose that $y$ lies in a weak nodal domain $\overline{D}$ different from $D$ and $D'$. Then the vertex $x\in D\cap D'\cap \overline{D}$, which is a contradiction by Proposition \ref{prop:1sphere}.
\end{proof}

\subsection{Strong nodal domains of the first eigenfunction}
It is natural to ask whether the strong nodal domain of the first eigenfunction of a symmetric matrix $M$ is the whole graph of $M$ or not. Recall that when the induced signed graph is balanced, the strong nodal domain of the first eigenfunction is the whole graph by Perron-Frobenius theorem. When the induced signed graph is non-balanced, we show in this section that either case can happen.
\begin{theorem}\label{Thm:nonzero}
	Let $\Gamma=( G,\sigma ) $ be a connected signed graph where $G=(V,E)$. Then there exists a symmetric matrix $M$ compatible with $\Gamma$, such that  the first eigenvalue $\lambda_1$ of $M$ is simple and the corresponding eigenfunction $f_1$ is nonzero everywhere. In particular, we have $\mathfrak{S}(f_1)=1$ and the strong nodal domain of $f_1$ is the whole graph $G$.
\end{theorem}
\begin{proof}
By Lemma \ref{lemma:2.8}, there exists a switching function $\tau: V\to \{+1,-1\}$ such that $\Gamma^\tau=(G,\sigma^\tau)$ has a spanning tree consisting of positive edges.
We consider the subgraph $\Gamma^\tau_+$ of $\Gamma^\tau$ with vertex set $V$ and edge set $\{\{x,y\}\in E: \sigma_{xy}=+1\}$.	Notice that $\Gamma^\tau_+$ is connected. Let $M_+$ by any symmetry matrix compatible with $\Gamma_+^\tau$. By Perron-Frobenius theorem, the first eigenvalue of $M_+$ is simple and its eigenfunction is positive on all vertices. We further consider the subgraph $\Gamma^\tau_-$ with vertex set $V$ and edge set $\{\{x,y\}\in E: \sigma_{xy}=-1\}$. Let $M_-$ be any symmetry matrix compatible with $\Gamma^\tau_-$. For any $\epsilon>0,\,\,M_++\epsilon M_-$ is compatible with $\Gamma^\tau$. By continuity of the eigenvalue and eigenfunction, the first eigenvalue $\lambda_{1}$ of $M_++\epsilon M_-$ is simple and the corresponding eigenfunction $f$ is positive when $\epsilon$ is small enough. Hence the matrix $D(\tau)(M_++\epsilon M_-)D(\tau)$ fulfills all requirements of the theorem.
	
By our construction, the signed graph $\Gamma$ has a spanning tree, whose walks are all strong nodal domain walks of the first eigenfunction $f_1$ of $D(\tau)(M_++\epsilon M_-)D(\tau)$. That is, the strong nodal domain of $f_1$ is the whole graph $G$.
\end{proof}

\begin{theorem}\label{Thm:zero vertex}
	Let $\Gamma=( G,\sigma ) $ be a non-balanced signed graph. Assume that there exists a vertex $z$ such that $\Gamma$ becomes balanced after removing $z$ and its incident edges. Then there exist a symmetric matrix $M$ compatible with $\Gamma$
	and an eigenfunction $f_1$ corresponding to its first eigenvalue, such that $f_1$ is zero at $z$ and non-zero elsewhere. In particular, the strong nodal domain of $f_1$ is not the whole graph $G$.
\end{theorem}
\begin{proof}
	Assume that $\Gamma$ has $n+1$ vertices which are denoted by $\{x_1,...,x_{n+1}\}$. We further assume that $\Gamma$ becomes balanced after removing $x_{n+1}$ and its incident edges. We denote by $\Gamma'$ the resulting balanced graph.
By the switching invariance, we can suppose that the sign of each edge in $\Gamma'$ is $+1$. Let $M'$ be any symmetric matrix compatible with $\Gamma'$. By Perron-Frobenius theorem, the first eigenvalue $\lambda_{1}$ of $M'$ is simple and the corresponding eigenfunction $f$ is positive.
	
We construct a symmetric matrix $M^\epsilon$ for any $\epsilon>0$ as follows:
\begin{equation}
	M^\epsilon_{ij}=\left\{
          \begin{array}{ll}
            M'_{ij}, & \hbox{$i,j$=1,2,\ldots,n;} \\
            \lambda_1+1, & \hbox{$i=j=n+1$;} \\
            \frac{\epsilon k_-}{f(x_i)}, & \hbox{$j=n+1\,\text{or}\,i=n+1,\,\,x_{i}\sim x_{n+1}\,\text{and}\,\sigma_{x_ix_{n+1}}=+1$;} \\
            -\frac{\epsilon k_+}{f(x_i)}, &\hbox{$j=n+1\,\text{or}\, i=n+1,\,  \,x_{i}\sim x_{n+1}\,\text{and}\,\sigma_{x_ix_{n+1}}=-1$;} \\
            0, & \hbox{otherwise,}
          \end{array}
        \right.
\end{equation}
where $k_+$ ($k_-$, resp.) is the number of the positive (negative, resp.) edges incident with $x_{n+1}$.
We define
\begin{equation}
	f_1(x_i)=\left\{
	\begin{array}{ll}
		f(x_i), &  \hbox{$i=1,2,\ldots,n$;}\\
		0, & \hbox{$i=n+1$.}
	\end{array} \right.
\end{equation}

By direct computation, we have $M^\epsilon f_1=\lambda_1 f_1$ for any $\epsilon>0$.  Recall that when $\epsilon=0$, $\lambda_{1}$ is the first eigenvalue of $M^\epsilon$. By continuity of the eigenvalue, $\lambda_{1}$ is still the first eigenvalue of $M^\epsilon$ when $\epsilon$ is small enough. Hence, the matrix $M^\epsilon$ with a small enough $\epsilon$ fulfills all requirements of the theorem.
\end{proof}


\section{Nodal domain theorems}\label{section:upper}
In this section, we prove the following nodal domain theorem. The proof is a neat extension of methods from \cite{DGLS01}. We will discuss its consequence for symmetric matrices whose induced signed graphs are trees via a duality argument.
\begin{theorem}\label{Thm:upper bound}
Let $M$ be a symmetric matrix, $\Gamma=(G,\sigma)$ be its induced signed graph where $G=(V,E)$, and $\lambda_k$ be its $k$-th eigenvalue. For any eigenfunction $f_k$ corresponding to $\lambda_k$, i.e., $Mf_k=\lambda_kf_k$, we have
\begin{equation}
\mathfrak{S}(f_k)\leq k+r-1,\,\,\text{and}\,\,\,\,\mathfrak{W}(f_k)\leq k+c-1,
\end{equation}
where $r$ is the multiplicity of $\lambda_{k}$ and $c$ is the number of connected components of $G$. In particular, when the graph $G$ is connected, we have $\mathfrak{W}(f_k)\leq k$.

If the eigenfunction $f_k$ corresponding to $\lambda_k$ has \emph{minimal support}, then we have
\begin{equation}\label{eq:minsupp}
\mathfrak{S}(f_k)\leq k.
\end{equation}
\end{theorem}
\begin{remark}
The estimate $\mathfrak{S}(f_k)\leq k+r-1$ has been proved by Mohammadian \cite{Mohammadian16}, see Remark \ref{rmk:Mohammadian}.
Mohammadian's proof using his estimate $c(\Gamma^{\leq}_{M,f_k}(G))\leq k$ for a connected graph $G$ and the interlacing theorem.  We give a direct argument below via the minimax principle.
\end{remark}
\begin{remark}
The estimate (\ref{eq:minsupp}) for the cases $k=1$ and $k=2$ is proved in \cite[Proposition 7 and Theorem 8]{Mohammadian16} extending a result of van der Holst \cite[Proposition 1]{vdH95}. We show it for any $k$ below. Recall from Lemma \ref{lemma:2.3}, any symmetric matrix has a basis consisting of eigenfunctions with minimal support. All the eigenfunctions $f_1,\ldots, f_6$ in Example \ref{ex:fish} have minimal supports.
\end{remark}

We prepare a crucial lemma, which is a reformulation of Duval and Reiner \cite[Lemma 5]{DR99}.
\begin{lemma}\label{Lemma:compute}
Let $M$ be a symmetric matrix. Let $\Gamma=(G,\sigma)$ be the induced signed graph of $M$ where $G=(V,E)$. Then for any two functions $f,g: V\to \mathbb{R}$, we have
\[\langle fg,M(fg) \rangle=\langle fg,fMg \rangle +\sum_{\{x,y\}\in E}(-M_{xy})g(x)g(y)(f(x)-f(y))^2.\]
\end{lemma}


\begin{proof}
By a direct calculation, we have
	\begin{equation*}
		\begin{aligned}
			&\langle fg,M(fg) \rangle\\
	        =&\sum_{x\in V}f(x)g(x)\left[\sum_{y\sim x}M_{xy}f(y)g(y)+M_{xx}f(x)g(x)\right]\\
	        =&\sum_{x\in V}f(x)g(x)\left[\sum_{y\sim x}\left(M_{xy}f(y)g(y)-M_{xy}f(x)g(y)+M_{xy}f(x)g(y)\right)+M_{xx}f(x)g(x)\right]\\
			=&\sum_{x\in V}f(x)g(x)\sum_{y\sim x}(-M_{xy})(f(x)-f(y))g(y)+\sum_{x\in V}f^2(x)g(x)\left[\sum_{y\sim x }M_{xy}g(y)+M_{xx}g(x)\right] \\
			=&\sum_{\{x,y\}\in E}(-M_{xy})g(x)g(y)(f(x)-f(y))^2+\langle fg,fMg\rangle.		
		\end{aligned}
	\end{equation*}
This completes the proof.
\end{proof}

The following corollary will be crucial for the proof of Theorem \ref{Thm:lower bound} in Section \ref{Section:lowerbd}.
\begin{corollary}\label{cor:XY}
Let $g$ be an eigenfunction of $M$ such that $Mg=\lambda_k g$. Let $D(g)$ be the diagonal matrix with $D(g)_{xx}=g(x)$. Then we have
\[\langle f,D(g)(M-\lambda_{k}I)D(g)f \rangle =\sum_{\{x,y\}\in E}(-M_{xy})g(x)g(y)(f(x)-f(y))^2.\]
\end{corollary}
\begin{proof}
By Lemma \ref{Lemma:compute}, we have
\[\langle fg,(M-\lambda_{k}I)fg \rangle =\langle fg, M(fg)-fMg \rangle=\sum_{\{x,y\}\in E}(-M_{xy})g(x)g(y)(f(x)-f(y))^2.\]
Then the corollary follows directly by observing that
\[\langle fg,(M-\lambda_{k}I)fg \rangle=\langle f,D(g)(M-\lambda_{k}I)D(g)f \rangle.\]
\end{proof}


For any $k$, let $f_k$ be the eigenfunction of $M$ corresponding to the $k$-th eigenvalue $\lambda_k$. Next, we show the estimates of $\mathfrak{S}(f_k)$ and $\mathfrak{W}(f_k)$ in Theorem \ref{Thm:upper bound}.

\begin{proof}[Proof of Theorem \ref{Thm:upper bound}: Estimates of $\mathfrak{S}(f_k)$]

Let $\{\Omega_i\}_{i=1}^m$ be the strong nodal domains of $f_k$, where $m=\mathfrak{S}(f_k)$.
For each $i$, we define
\begin{equation}\label{eq:gi}
g_{i}(x)=\left\{
           \begin{array}{ll}
             f_k(x), & \hbox{if $x\in \Omega_i$;} \\
             0, & \hbox{otherwise.}
           \end{array}
         \right.
\end{equation}
Since $g_i, i=1,2,\ldots,m,$ are linearly independent, we can find $a_i\in \mathbb{R}, i=1,2,\ldots,m$,\,such that the function $g:=\sum_{i=1}^{m}a_ig_i$ satisfies
\[\langle g,f_i \rangle=0,\,\,\text{for}\,\,i=1,\ldots,m-1.\]

We introduce a function $a:V \to \mathbb{R}$ defined as
$a(x)=a_i$ if $x\in \Omega_i$ for some $i$ and $a(x)=0$ otherwise. Then we can write $g=af_k$. Applying Lemma \ref{Lemma:compute} yields

 \begin{equation*}
 	\begin{aligned}
 	\langle g,Mg \rangle
 		&= \langle af_k,aMf_k\rangle+\sum_{\{x,y\})\in E}(-M_{xy})f_k(x)f_k(y)(a(x)-a(y))^2\\
 		&= \lambda_{k}\langle g,g \rangle+\sum_{\{x,y\}\in E}(-M_{xy})f_k(x)f_k(y)(a(x)-a(y))^2.
 	\end{aligned}
 \end{equation*}
 By Lemma \ref{Lemma:minmax}, we derive
 \begin{equation*}
 	\lambda_m\leq \frac{\langle g,Mg \rangle}{\langle g,g\rangle}\leq \lambda_{k}+\frac{1}{{\langle g,g \rangle}}\sum_{\{x,y\}\in E}(-M_{xy})f_k(x)f_k(y)(a(x)-a(y))^2.
 \end{equation*}
For each edge $\{x,y\}\in E$,
if $(-M_{xy})f_k(x)f_k(y)>0$, then $f_k(x)\sigma_{xy}f_k(y)>0$. That is, the vertices $x$ and $y$ lie in the same strong nodal domain. Hence, $a(x)-a(y)=0$ and $(-M_{xy})f_k(x)f_k(y)(a(x)-a(y))^2=0$. If, otherwise, $(-M_{xy})f_k(x)f_k(y) \leq 0$, we have $(-M_{xy})f_k(x)f_k(y)(a(x)-a(y))^2\leq 0$. Therefore, we obtain
\[\sum_{\{x,y\}\in E}(-M_{xy})f_k(x)f_k(y)(a(x)-a(y))^2\leq 0.\]
This leads to $\lambda_m\leq \lambda_k$. Recall that $\lambda_{k}<\lambda_{k+r}$, we have $\mathfrak{S}(f_k)=m\leq k+r-1$.

Suppose that $f_k$ has minimal support. Next we show $\mathfrak{S}(f_k)\leq k$.
We prove that by contradiction. Assume that $\mathfrak{S}(f_k)>k$. We construct a nonzero function $g:=\sum_{i=1}^k a_ig_i$ as in the above such that $\langle g,f_i\rangle=0$ for $i=1,2,\ldots, k-1$. By construction, we have $\mathfrak{S}(g)\leq k$. By the same argument as in the above, we derive
\[\lambda_k=\frac{\langle g, Mg\rangle}{\langle g, g\rangle}.\]
Therefore, $g$ is an eigenfunction corresponding to $\lambda_k$. However, we have $\supp(g)\varsubsetneq \supp(f)$ which is a contradiction.
\end{proof}

In order to show the estimate of $\mathfrak{W}(f_k)$, we first prepare the following discrete unique continuation lemma for eigenfunctions.
\begin{lemma}\label{lemma:unique} Consider an eigenfunction $f$ of a symmetric matrix $M$ corresponding to an eigenvalue $\lambda$. Let $\{D_i\}_{i=1}^m$ be the set of weak nodal domains of $f$. For each $i\in \{1,\ldots,m\}$, define a function $g_i$ as below
\[g_i(x)=\left\{
               \begin{array}{ll}
                 f(x), & \hbox{if $x\in D_i$;} \\
                 0, & \hbox{otherwise.}
               \end{array}
             \right.
\]
Assume that the induced signed graph $\Gamma=(G,\sigma)$ of $M$ is connected. If the function \[g:=\sum_{i=1}^ma_ig_i,\,\, \text{where} \,\,a_i\in \mathbb{R},\,i=1,2,\ldots,m,\] is an eigenfunction of $M$ corresponding to $\lambda$, then we have
\[a_1=a_2=\cdots=a_m.\]
\end{lemma}
\begin{remark}
We have the following interesting consequence: For those functions $g$ defined as above, if $g=f$ on one weak nodal domain, then $g=f$ everywhere.
\end{remark}
\begin{proof}[Proof of Lemma \ref{lemma:unique}]
We first observe that $g$ can be reformulated as a product of two functions, $g=af$, where the function $a: V\to \mathbb{R}$ is defined as $a(x)=a_i$ if $x\in D_i$ and $f(x)\neq 0$, and $a(x)=0$ otherwise.

For any adjacent $D_i$ and $D_j$, we can always find $x_0 \in D_i,\,y_0\in D_j\setminus D_i$ such that $x_0\sim y_0$. Indeed, when $D_i\cap D_j\neq \emptyset$, due to the connectedness of $D_j$, there exists a path connecting any $x_i\in D_i\cap D_j$ and any $z\in D_j\setminus D_i$.

Next, we show $a_i=a_j$ for any two adjacent weak nodal domains $D_i$ and $D_j$ if $a_i\neq 0$. We divide our arguments into two cases.

\textbf{Case 1}: $f(x_0)\neq 0$. By Lemma \ref{Lemma:compute}, we have
\begin{equation}\label{eq:m2}
 \lambda=\frac{\langle g, Mg\rangle}{\langle g,g\rangle}=\lambda+\frac{1}{\langle g,g\rangle}\sum_{\{x,y\}\in E}(-M_{xy})f(x)f(y)(a(x)-a(y))^2.
\end{equation}
For each edge $\{x,y\}\in E$, if $(-M_{xy})f(x)f(y)>0$, then $f(x)\sigma_{xy}f(y)>0$. That is, the vertices $x$ and $y$ lie in the same weak nodal domain with both $f(x)$ and $f(y)$ nonzero, which means $a(x)-a(y)=0$. Therefore, we obtain
\begin{equation}\label{eq:nonnegative1}
(-M_{xy})f(x)f(y)(a(x)-a(y))^2\leq 0,\,\,\text{for any }\,\,\{x,y\}\in E,
\end{equation}
Combining (\ref{eq:m2}) and (\ref{eq:nonnegative1}) yields that for any $\{x,y\}\in E$,
 \begin{equation}\label{eq:zero}
	-M_{xy}f(x)f(y)(a(x)-a(y))^2= 0.
\end{equation}
For the edge $\{x_0,y_0\}$, we have  $f(y_0)\neq 0$ since $y_0\in D_j\setminus D_i$. Moreover, we have \[-M_{x_0y_0}f(x_0)f(y_0)< 0,\] since $x_0,y_0$ lies in two different weak nodal domains with both $f(x_0)$ and $f(y_0)$ nonzero. By (\ref{eq:zero}), we have $a(x_0)=a(y_0)$. Since $f(x_0)\neq 0$ and $f(y_0)\neq 0$, we have $a_i=a_j$.

\textbf{Case 2}: $f(x_0)=0$.
By Corollary \ref{cor:3.8}, we have $S_1(x_0):=\{v\in V: v\sim x_0\}\subset D_i\cup D_j$.\\
We define a function $h:=f-\frac{1}{a_i}g$. Observe that $h|_{D_i}=0$, and $h$ is an eigenfunction of $M$ corresponding to $\lambda$.
So we have
\begin{align*}
	0=&-\lambda h(x_0)=Mh(x_0)=\sum_{y\sim x_0}M_{x_0y}h(y)\\
    =&\sum_{\substack{y\sim x_0\\ y\in D_j\setminus D_i}}M_{x_0y}h(y)=\left(1-\frac{a_j}{a_i}\right)\sum_{\substack{y\sim x_0\\ y\in D_j\setminus D_i}}M_{x_0y}f(y).
\end{align*}
By the definition of the weak nodal domain, it holds that
\[M_{x_0y}f(y)M_{x_0y'}f(y')>0,\,\,\text{for any}\,\,y,y'\sim x_0 \,\,\text{and}\,\,y,y'\in D_j\setminus D_i.\]
Therefore, we have
\begin{equation*}
	\sum_{\substack{y\sim x_0\\ y\in D_j\setminus D_i}}M_{x_0y}f(y)\neq 0.
 \end{equation*}
This tells that $a_i=a_j$.

Since $g$ is an eigenfunction, at least one of $a_i, \,i=1,\ldots, m$ is non-zero. Then the lemma follows directly from the above argument and the connectedness from Proposition \ref{prop:connect}.
\end{proof}
\begin{proof}[Proof of Theorem \ref{Thm:upper bound}: Estimates of $\mathfrak{W}(f_k)$]

We first assume that the signed graph induced by $M$ is connected. We denote all weak nodal domains of $f_k$ by $\{D_i\}_{i=1}^m$, where $m=\mathfrak{W}(f_k)$. We introduce for each $i$
\begin{equation*}
g_{i}(x)=\left\{
           \begin{array}{ll}
             f_k(x), & \hbox{if $x\in D_i$;} \\
             0, & \hbox{otherwise.}
           \end{array}
         \right.
\end{equation*}
Let $a_i\in \mathbb{R}, i=1,2,\ldots,m,$ be $m$ constants such that $g:=\sum_{i=1}^{m}a_ig_i$ satisfies $\langle g,f_i \rangle=0$ for $i=1,....m-1$. We define a function $a:V \to \mathbb{R}$ as $a(x)=a_i$ if $x\in D_i$ and $f_k(x)\neq 0$, and $a(x)=0$ otherwise. By construction, $g=af_k$. We then derive
 \begin{equation}\label{eq:mkRayleigh}
	\lambda_m\leq \frac{\langle g,Mg \rangle}{\langle g,g\rangle}\leq \lambda_{k}+\frac{1}{{\langle g,g \rangle}}\sum_{\{x,y\}\in E}(-M_{xy})f_k(x)f_k(y)(a(x)-a(y))^2.
\end{equation}
Similarly as in the proof of (\ref{eq:nonnegative1}), we have
\begin{equation}\label{eq:nonnegative}
(-M_{xy})f_k(x)f_k(y)(a(x)-a(y))^2\leq 0,\,\,\text{for any }\,\,\{x,y\}\in E,
\end{equation}
and, hence,
\begin{equation}\label{eq:mk}
\lambda_m\leq \lambda_k.
\end{equation}
We argue by contradiction. Suppose $m>k$, then $\lambda_{m}\geq \lambda_{k}$. By (\ref{eq:mkRayleigh}) and (\ref{eq:mk}), we have
\begin{equation}
	\lambda_{m}=\frac{\langle g,Mg \rangle}{\langle g,g \rangle}= \lambda _{k}.
\end{equation}
Hence, $g$ is an eigenfunction of $M$ corresponding $\lambda_{k}$.

Then we can apply Lemma \ref{lemma:unique} to show that $g=af_k$ where $a$ is a nonzero constant function. However, $\langle g,f_{j} \rangle=0$ for any $j<m$ by construction. Our assumption $m>k$ then implies $\langle g,f_k \rangle=0$. That is, $a\langle f_k,f_k \rangle=0$. This is a contradiction. So we get
\begin{equation}\label{eq:weaknodal}
\mathfrak{W}(f_k)=m\leq k.
\end{equation}

In general, we denote by $\{\Gamma_i\}_{i=1}^c$  the $c$ connected components of the signed graph induced by $M$. Let $M^i, f_k^i$ be the restriction of $M, f_k$ to the connected component $\Gamma^i$ respectively. Then either $f_k^i$ is identically zero on $\Gamma_i$ or $M^i f_k^i=\lambda_{k_i}^i f_k^i$, where $\lambda_{k_i}^i=\lambda_{k}$ is the $k_i$-th eigenvalue of $M^i$. Moreover, we can assume $\lambda^i_{k_i-1}<\lambda^i_{k_i}$. Without loss of generality, we assume $\{\Gamma_i\}_{i=1}^\ell$ be the connected components on which $f_k$ is not identically zero. Employing the fact (\ref{eq:weaknodal}) we estimate
	\[\mathfrak{W}(f_k)\leq\sum_{i=1}^\ell  k_i\leq \sum_{i=1}^\ell (k_i-1)+\ell<k+\ell\leq k+c. \]
This completes the proof.
\end{proof}

\begin{corollary}\label{cor:fiedler}
Let $M$ be an $n\times n$ symmetric matrix with the induced signed graph $\Gamma=(T,\sigma)$, where $T$ is a tree. If $f_k$ is an eigenfunction corresponding to the $k$-th eigenvalue $\lambda_k$ of $M$ which is not zero at any vertex, then
 \begin{itemize}
   \item [(i)] $\lambda_k$ is simple;
   \item [(ii)]$\mathfrak{S}(f_k)=k$.
 \end{itemize}
Consequently, any eigenfunction corresponding to a multiple eigenvalue of such a matrix $M$ vanishes on at least one vertex.
\end{corollary}
Corollary \ref{cor:fiedler} is due to Fiedler \cite[(2,5) Corollary, (2,6) Corollary]{Fiedler75} and B\i y\i ko\u{g}lu \cite[Theorem 2]{B03}. We give an alternative proof here as a consequence of Theorem \ref{Thm:upper bound} and the Theorem \ref{Thm:duality}.
\begin{proof}
By Theorem \ref{Thm:duality}, we have
\begin{equation}\label{eq:1}
\mathfrak{S}(f_k)+\overline{\mathfrak{S}}(f_k)=n+1.
\end{equation}
Since $f_k$ has no zeros, we have $\mathfrak{S}(f_k)=\mathfrak{W}(f_k)$. By Theorem \ref{Thm:upper bound}, we obtain
\begin{equation}\label{eq:2}
\mathfrak{S}(f_k)\leq k.
\end{equation}
Without loss of generaliy, let $\lambda_k$ be the eigenvalues of $M$ with multiplicity $r$ such that
\[\lambda_1\leq\cdots\leq\lambda_{k-1}<\lambda_k=\cdots=\lambda_{k+r-1}<\lambda_{k+r}\leq\cdots\leq \lambda_n.\]

	Observe that the signed graph induced by $-M$ is the negation $-\Gamma=(G,-\sigma)$ of $\Gamma=(G,\sigma)$.
	The eigenvalues of $-M$ can be listed accordingly as below
		\[-\lambda_{n}\leq\cdots\leq-\lambda_{k+r}<-\lambda_{k+r-1}=\cdots=-\lambda_{k}<-\lambda_{k-1}\leq\cdots\leq-\lambda_{1}.\]
Notice that $f_k$ is an eigenfunction of $-M$ corresponding to $-\lambda_{k}$.
Similarly, we derive
\begin{equation}\label{eq:3}
\overline{\mathfrak{S}}(f_k)\leq n-(k+r-1)+1.
\end{equation}
Combining the estimates (\ref{eq:1}), (\ref{eq:2}), and (\ref{eq:3}) yields
\[n+1=\mathfrak{S}(f_k)+\overline{\mathfrak{S}}(f_k)\leq n-r+2.\]
This tells that $r=1$. Therefore, the inequalities in (\ref{eq:2}) and (\ref{eq:3}) are both equality. In particular, we have $\mathfrak{S}(f_k)=k$.
\end{proof}

\section{A reformulation of Fiedler's approach on acyclic matrices}\label{Section:Fiedler}
Fiedler \cite{Fiedler75} studied the eigenvectors of \emph{acyclic matrices}, i.e., symmetric matrices whose induced graph does not contain any cycle.  In this section, we explain that Fiedler's result can be reformulated as estimates for the number of \emph{strong nodal domains} of eigenfunctions for acyclic matrices in our terminology.

Observe that the induced signed graph $\Gamma=(T,\sigma)$ of any acyclic matrix is always balanced, since the graph $T$ has no cycles. That is, $T$ is a forest.

We first quote the main theorem of Fiedler \cite[(2,3) Theorem]{Fiedler75} below.
\begin{theorem}[\cite{Fiedler75}]\label{thm:Fiedler75}
Let $A=(a_{ik})$ be an $n\times n$ acyclic matrix. Let $y=(y_i)$ be an eigenvector of $A$ corresponding to an eigenvalue $\lambda$. Denote by $\omega^+$ and $\omega^-$, respectively, the number of eigenvalues of $A$ greater than and smaller than $\lambda$, and let $\omega^{(0)}$ be the multiplicity of $\lambda$.

Let there be first no "isolated" zero coordinate of $y$, i.e. coordinate $y_k=0$ such that $a_{kj}y_j=0$ for all $j$. Then
\begin{equation}\label{eq:4}
\omega^+=a^++m,\,\,\omega^-=a^-+m, \,\,\omega^{(0)}=n-\omega^+-\omega^-
\end{equation}
where $m$ is the number of zero coordinates of $y$, $a^+$ is the number of those (unordered) pairs $(i,k), \,i\neq k$, for which
\[a_{ik}y_iy_k<0\]
and $a^-$ is the number of such pairs $(i,k),\, i\neq k$, for which
\[a_{ik}y_iy_k>0.\]
If there are isolated zero coordinates of $y$, if $\mathcal{F}$ is the set of indices corresponding to such coordinates and $\widetilde{A}$ the matrix obtained from $A$ by deleting all rows and columns with indices from $\mathcal{F}$ then the numbers $\omega^+, \omega^-$ and $\omega^{(0)}$ satisfy
\begin{equation}\label{eq:5}
\omega^+=\tilde{\omega}^++c_1,\,\,\omega^-=\tilde{\omega}^-+c_2,\,\,\omega^{(0)}=\tilde{\omega}^{(0)}+c_0
\end{equation}
where $\tilde{\omega}$ are corresponding numbers of $\widetilde{A}$ and $c_0,c_1,c_2$ nonnegative integers such that
\begin{equation}\label{eq:6}
c_0+c_1+c_2=|\mathcal{F}|,
\end{equation}
the number of elements in $\mathcal{F}$.
\end{theorem}
With our terminology, we reformulate Theorem \ref{thm:Fiedler75} as below.

\begin{theorem}\label{thm:reformulation}
	Let $A=(a_{ik})$ be an $n\times n$ acyclic matrix and $\Gamma=(T,\sigma)$ where $T=(V,E)$ be the induced signed graph. Let $\lambda_{k}$ be the $k$-th eigenvalue of $M$ with multiplicity $r$ and eigenfunction $f_k$ such that
\[\lambda_1\leq\cdots\leq\lambda_{k-1}<\lambda_k=\cdots=\lambda_{k+r-1}<\lambda_{k+r}\leq\cdots\leq \lambda_n.\]
We denote by
\begin{equation}\label{eq:Fiedlerset}
\mathcal{F}:=\{x\in V: f_k(x)=0 \,\,\text{and}\,\,f_k(y)=0\,\,\text{for all}\,\,y\sim x\}
\end{equation} the set of Fiedler's isolated zeros of $f_k$, and by $\overline{\mathfrak{S}}(f_k)$ the number of strong nodal domains of $f_k$ on the signed graph $-\Gamma=(T,-\sigma)$. Let $\tilde{r}$ be the multiplicity of $\lambda_k$ as an eigenvalue of the matrix $\widetilde{A}$ obtained from $A$ by deleting all rows and columns with indices from $\mathcal{F}$.
Then we have $r\geq \tilde{r}$ and
  \begin{equation}\label{eq:55}
  \mathfrak{S}(f_k)= k+r-1-|\mathcal{F}|+c_1,\,\,\text{and}\,\,\,\overline{\mathfrak{S}}(f_k)= n-k+1-|\mathcal{F}|+c_2,
  \end{equation}
  where $c_1,c_2$ are nonnegative integers such that
  $c_1+c_2+(r-\tilde{r})=|\mathcal{F}|.$
In particular, when $\mathcal{F}=\emptyset$, we have
  \begin{equation}\label{eq:44}
  \mathfrak{S}(f_k)=k+r-1,\,\,\text{and}\,\,\,\overline{\mathfrak{S}}(f_k)=n-k+1.
  \end{equation}
\end{theorem}
\begin{proof}We show how to derive the theorem from Fiedler's Theorem \ref{thm:Fiedler75}. By definition, we have
\begin{equation}\label{eq:omega}
  \omega^+=n-(k+r-1),\,\,\text{and}\,\,\,\omega^-=k-1.
\end{equation}
Observe that for any edge $\{i,j\}\in E$, $a_{ij}f_k(i)f_k(j)<0$ if and only if $f_k(i)\sigma_{ij}f_k(j)>0$ since $\sigma_{ij}=-a_{ij}/|a_{ij}|$. Then applying Lemma \ref{lemma:strong}, we have
\begin{equation}\label{eq:a+}
a^+=n-z-\mathfrak{S}(f_k),
\end{equation}
where $z$ is the number of zeros of $f_k$. Similarly, we obtain
\begin{equation}\label{eq:a-}
a^-=n-z-\overline{\mathfrak{S}}(f_k).
\end{equation}

If $\mathcal{F}=\emptyset$, combining (\ref{eq:4}) with (\ref{eq:omega}), (\ref{eq:a+}) and (\ref{eq:a-}) yields
\[n-(k+r-1)=n-\mathfrak{S}(f_k),\,\,\text{and}\,\,\,k-1=n-\overline{\mathfrak{S}}(f_k).\]
That is, (\ref{eq:44}) holds true.

If $\mathcal{F}\neq \emptyset$, let $\widetilde{A}$ be the matrix obtained from $A$ by deleting all rows and columns with indices from $\mathcal{F}$. Then $f_k$ restricting to the set $V\setminus \mathcal{F}$ is an eigenfunction of $\widetilde{A}$ corresponding to the eigenvalue $\lambda_k$. Suppose $\lambda_k=\mu_{\tilde{k}}$ such that all the eigenvalues of $\widetilde{A}$ can be listed as
\[\mu_1\leq\cdots\leq\mu_{\tilde{k}-1}<\mu_{\tilde{k}}=\cdots=\mu_{\tilde{k}+\tilde{r}-1}<\mu_{\tilde{k}+\tilde{r}}\leq\cdots\leq \mu_{n-|\mathcal{F}|}.\]
Then we have
\begin{equation}\label{eq:omegatilde}
\tilde{\omega}^+=n-|\mathcal{F}|-(\tilde{k}+\tilde{r}-1),\,\,\text{and}\,\,\,\tilde{\omega}^-=\tilde{k}-1.
\end{equation}
By (\ref{eq:44}), we have
\begin{equation}\label{eq:tilde}
\mathfrak{S}(f_k)=\mathfrak{S}(f_k|_{V\setminus\mathcal{F}})=\tilde{k}+\tilde{r}-1,\,\,\text{and}\,\,\,\overline{\mathfrak{S}}(f_k)=\overline{\mathfrak{S}}(f_k|_{V\setminus\mathcal{F}})=n-|\mathcal{F}|-\tilde{k}+1.
\end{equation}
Inserting (\ref{eq:omega}), (\ref{eq:omegatilde}) and (\ref{eq:tilde}) into (\ref{eq:55}) leads to
\[n-(k+r-1)=n-|\mathcal{F}|-\mathfrak{S}(f_k)+c_1,\,\,\text{and}\,\,\,k-1=n-|\mathcal{F}|-\overline{\mathfrak{S}}(f_k)+c_2,\]
which confirms (\ref{eq:55}).
\end{proof}
\begin{corollary}\label{cor:5.3}
Let $A=(a_{ik})$ be an $n\times n$ acyclic matrix and $\Gamma=(T,\sigma)$ with $T=(V,E)$ be the induced signed graph. Let $\lambda_{k}$ be the $k$-th eigenvalue of $M$ with multiplicity $r$ and eigenfunction $f_k$ such that
\[\lambda_1\leq\cdots\leq\lambda_{k-1}<\lambda_k=\cdots=\lambda_{k+r-1}<\lambda_{k+r}\leq\cdots\leq \lambda_n.\]
Denote $\mathcal{F}:=\{x\in V: f_k(x)=0 \,\,\text{and}\,\,f_k(y)=0\,\,\text{for any}\,\,y\sim x\}$. Then we have
\[k+r-1-|\mathcal{F}|\leq \mathfrak{S}(f_k)\leq k+r-1,\,\,\text{and}\,\,\,n-k+1-|\mathcal{F}|\leq\overline{\mathfrak{S}}(f_k)\leq n-k+1.\]
\end{corollary}
\begin{proof}
Since $c_1+c_2+(r-\tilde{r})=|\mathcal{F}|$ and $r\geq \tilde{r}$, we have $0\leq c_1, c_2\leq |\mathcal{F}|$. Then, the Corollary is an immediately consequence of Theorem \ref{thm:reformulation}.
\end{proof}
Combining with Theorem \ref{Thm:duality}, we derive the following consequence.
\begin{corollary}\label{cor:multiplicitytilde}
Let $A=(a_{ik})$ be an $n\times n$ acyclic matrix and $\Gamma=(T,\sigma)$ with $T=(V,E)$ be the induced signed graph. Let $\lambda$ be an eigenvalue of $M$ with multiplicity $r$ and eigenfunction $f: V \to \mathbb{R}$. Denote $\mathcal{F}:=\{x\in V: f(x)=0 \,\,\text{and}\,\,f(y)=0\,\,\text{for any}\,\,y\sim x\}$. Let $\tilde{r}$ be the multiplicity of $\lambda$ as an eigenvalue of $\widetilde{A}$ which is the matrix obtained from $A$ by deleting all rows and columns with indices from $\mathcal{F}$.
Then, we have
\begin{equation}
r\geq \tilde{r}=e_0-2z+c+|\mathcal{F}|,
\end{equation}
where $z=|\{x\in V: f(x)=0\}|$, $e_0=|\{\{x,y\}\in E: f(x)=0 \,\,\text{or}\,\,f(y)=0\}|$, and $c$ is the number of connected components of $T$.
\end{corollary}
\begin{proof}
On one hand, Theorem \ref{thm:reformulation} implies that
\[\mathfrak{S}(f)+\overline{\mathfrak{S}}(f)=n+r-2|\mathcal{F}|+c_1+c_2=n+r-|\mathcal{F}|-(r-\tilde{r})=n-|\mathcal{F}|+\tilde{r}.\]
On the other hand, we have by Theorem \ref{Thm:duality} that
\[\mathfrak{S}(f)+\overline{\mathfrak{S}}(f)=n+c-2z+e_0.\]
Combining the above two identities leads to
$\tilde{r}=e_0-2z+c+|\mathcal{F}|.$
\end{proof}
\begin{remark}
Suppose that the eigenfunction $f$ of the acyclic matrix $A$ satisfies the following property: There are no two indices $i,j$ such that $a_{ij}\neq 0$ and $f(i)=f(j)=0$. Then we have $\mathcal{F}=\emptyset$, $e_0=\sum_{x\in V: f(x)=0}d_x$, and
\[r=\tilde{r}=\sum_{x\in V:f(x)=0}d_x-2z+c=c+\sum_{x\in V:f(x)=0}(d_x-2).\]
The above identity has been shown in \cite[(2,4) Theorem]{Fiedler75}. Our Corollary \ref{cor:multiplicitytilde} is, therefore, an extension of \cite[(2,4) Theorem]{Fiedler75}.
\end{remark}
\section{Lower bounds of the number of strong nodal domains}\label{Section:lowerbd}
We discuss in this section two lower bound estimates of the number of strong nodal domains and related applications.

We first prepare some notations. A sequence $\{x_i\}_{i=1}^k$ of vertices in a signed graph $\Gamma=(G,\sigma)$ is called \emph{a strong nodal domain cycle} (S-cycle for short) of a function $f$ on $\Gamma$ if it is an S-walk of $f$ and $x_k=x_1$.
\begin{definition}
Let $G=(V,E)$ be a graph. We define \[\ell(G):=|E|-|V|+c(G),\] where $c(G)$ is the number of connected components of $G$. Let $\Gamma=(G,\sigma)$ be a signed graph and $f: V\to \mathbb{R}$ be a function. Let $H$ be a graph whose vertex set $V(H)=V$ and edge set $E(H):=\{\{x,y\}\in E: f(x)\sigma_{xy}f(y)>0\}$. We define
\[\ell_+(G,\sigma,f):=|E(H)|-|V(H)|+c(H)\]
where $c(H)$ is the the number of connected components of $H$.
\end{definition}
\begin{remark}
\begin{itemize}
  \item [(i)] The number $\ell(G)$ is the minimal number of edges that need to be removed from $G$ in order to turn it into a forest. It is the dimension of the cycle space of $G$ \cite[Corollary 1.33]{BL95}.
  \item [(ii)] The number of $\ell_+(G,\sigma,f)$ is the dimension of the vector space of $S$-cycles of the function $f$ on the singed graph $(G,\sigma)$.
\end{itemize}
\end{remark}

\begin{definition}
Let $G=(V,E)$ be a graph. A vertex $x\in V$ is called a \emph{tree-like vertex} if removing $x$ and its incident edges from $G$ increases the number of connected components by $d_x-1$.
\end{definition}
All vertices in a forest are tree-like. Moreover, we have the observations below.
\begin{proposition}\label{prop:treelike}
The tree-like vertices have the following properties:
\begin{itemize}
  \item [(i)] Let $x$ be a tree-like vertex in a graph $G=(V,E)$. Let $G\setminus\{y\}$ be a graph obtained from $G$ by removing any vertex $y\neq x$ and its incident edges. Then $x$ is still a tree-like vertex in the graph $G\setminus\{y\}$.
  \item [(ii)] Let $G'$ be the induced subgraph of $G=(V,E)$ on the set $V\setminus Y$ where $Y$ is a set of tree-like vertices. Then we have $\ell(G')=\ell(G)$.
\end{itemize}
\end{proposition}
\begin{proof}
The property (i) follows directly from the observation that a vertex is tree-like if and only if it belongs to no cycle. Moreover, for any $y\in Y$, we have
\[\ell(G\setminus\{y\})=(|E|-d_y)-(|V|-1)+(c(G)+d_y-1)=\ell(G).\]
By (i), we can apply the above argument iteratively to conclude (ii).
\end{proof}

Next, we define the following particular set of zeros of a function on a graph.
\begin{definition}\label{def:Fiedler zero set}
Let $G=(V,E)$ be a graph and $f:V\to \mathbb{R}$ be a function. We define the \emph{Fiedler zero set} of $f$ on $G$ as below:
\[\mathcal{F}(G,f):=\{x\in V: f(x)=0, \,\,\text{and either}\,\, f(y)=0\,\,\text{for all}\,\,y\sim x,\,\,\text{or}\,\,x\,\,\text{is not tree-like}\}.\]
\end{definition}
When $G$ is a forest, the set $\mathcal{F}(G,f)$ coincides with Fiedler's set of isolated zeros (\ref{eq:Fiedlerset}).
We denote by $\mathcal{F}^c(G,f)$ the complement of $\mathcal{F}(G,f)$ in the zero set of $f$, i.e.,
 \[\mathcal{F}^c(G,f):=\{x\in V: f(x)=0, x\,\,\text{is tree-like and there exists}\,\,y\sim x\,\,\text{such that}\,\,f(y)\neq 0\}.\]

Now we are ready to state our first lower bound estimate.
\begin{theorem}\label{Thm:lower bound}
	Let $M$ be an $n\times n$ symmetric matrix and $\Gamma=(G,\sigma)$ be the induced signed graph. Let $\lambda_{k}$ be the $k$-th eigenvalue of $M$ with multiplicity $r$ and eigenfunction $f_k$ such that
\[\lambda_1\leq\cdots\leq\lambda_{k-1}<\lambda_k=\cdots=\lambda_{k+r-1}<\lambda_{k+r}\leq\cdots\leq \lambda_n.\]
Then we have
\[\mathfrak{S}(f_k)\geq k+r-1-\ell'+\ell_+-|\mathcal{F}|,\]
where $\mathcal{F}=\mathcal{F}(G,f_k)$ is the Fielder zero set, $\ell_+=\ell_+(G,\sigma,f_k)$ is the dimension of the S-cycle space of $G$, and $\ell'=\ell(G')$ is the dimension of the cycles space of $G'$, where $G'$ is the induced subgraph of $G$ on the set of nonzeros $V\setminus\{x: f_k(x)=0\}$.
\end{theorem}

\begin{remark}The above lower bound estimate is an extension of previous works in \cite{Fiedler75,B03,Berkolaiko08,XY12,Mohammadian16} with improvements. Xu and Yau \cite[Theorem 1.3]{XY12} have shown for generalized Laplacians that $\mathfrak{S}(f_k)\geq k+r-\ell-z$, where $\ell=\ell(G)$ and $z=|\mathcal{F}|+|\mathcal{F}^c|$ is the total number of zeros. Since $\ell'\leq \ell$, $\ell_+\geq 0$ and $z\geq |\mathcal{F}|$, Theorem \ref{Thm:lower bound} improves Xu and Yau's estimate. Moreover, when $M$ is acyclic, Theorem \ref{Thm:lower bound} reduce to the lower bound estimates in Corollary \ref{cor:5.3}.
\end{remark}
The proof is built upon techniques developed in \cite{Fiedler75} and \cite{XY12}.

Let us prepare four Lemmas. We first recall the following results of Fiedler \cite{Fiedler75}.
\begin{lemma}\cite[(1,8)]{Fiedler75}\label{Lemma:Fiedler}
	Let $T$ be a tree with the set of vertices $\{1,\ldots,n\}$. Then the $n-1$ linear forms $x_i-x_j$, where $\{i,j\},\,i<j$, are edges of $T$, are linearly indepenent.
\end{lemma}
\begin{lemma}\cite[(1,12) Lemma]{Fiedler75}\label{lemma:Fiedler2}
Let \[A=\left(
        \begin{array}{cc}
          B & a \\
          a^T & b \\
        \end{array}
      \right)
\]
be an $n\times n$ partitioned symmetric matrix where $b\in \mathbb{R}$. If there exists $u\in \mathbb{R}^{n-1}$ such that $Bu=0$ and $a^Tu\neq 0$. Then we have
\[p_A=p_B+1,\]
where $p_A$ and $p_B$ are the positive indices of inertia of $A$ and $B$, respectively.
\end{lemma}
The following result is the so-called \emph{interlacing theorem}. For its proof and a brief historical review, we refer to \cite[Section 2]{Haemers95}.
\begin{lemma}\label{thm:interlacing}
Let $A$ be a Hermitian matrix with eigenvalues $\lambda_{1}\leq \cdots\leq \lambda_{n}$ and $B$ be a principle submatrix of $A$ with eigenvalues $\mu_1\leq\cdots\leq\mu_m$. Then we have the inequalities \[\lambda_{i}\leq \mu_i\leq \lambda_{n-m+i},\]
for any $1\leq i \leq m$.
\end{lemma}
The following linear algebraic lemma is essentially taken from \cite[Lemma 2.3]{XY12}.
\begin{lemma}\label{Lemma:5.3}
Consider a quadratic form \[B=\sum_{i,j=1}^na_{ij}(x_i-x_j)^2, \,\,\text{where}\,\,a_{ij}=a_{ji}\in \mathbb{R}.\]
Let $G=(V,E)$ be the graph with $V=\{1,2,\ldots,n\}$ and $E=\{\{i,j\}: a_{ij}\neq 0 \,\text{and}\, i\neq j\}$.
Let $H=(V,E(H))$ with $E(H)=\{\{i,j\}\in E: a_{ij}>0\}$ be a subgraph of $G$.
For any spanning forest $T$ of $H$, we have
	\begin{equation}\label{eq:edge}
		p\leq |E(T)|\leq |E(H)|\leq p+\ell,
	\end{equation}
where $p$ is the positive index of inertia of $B$, and $\ell=\ell(G)$.
\end{lemma}

\begin{remark}
\begin{itemize}
  \item [(i)]When the corresponding graph $G$ is a forest, the linear forms $x_i-x_j$, where $\{i,j\}\in E$ are linearly independent by Lemma \ref{Lemma:Fiedler}. Then the fact $|E(H)|=p$ follows directly from Sylvester's law of inertia.
  \item [(ii)] Let $A$ be an $n\times n$ symmetric matrix such that $A_{ij}=-a_{ij}$ for any $i\neq j$, and $A\mathbf{1}=0$, where $\mathbf{1}=(1,\ldots,1)^T$. Then the positive index of inertia of the matrix $A$ is equal to that of the quadratic form $B$, since $B=-2\sum_{i,j=1}^na_{ij}x_ix_j$. Therefore, (\ref{eq:edge}) tells $p_A\leq |E(T)|\leq |E(H)|\leq p_A+\ell$. Notice that $2|E(H)|$ is equal to the number of negative off-diagonal entries of $A$. For the case $\ell=0$, i.e., $A$ is acyclic, this has been shown in \cite[(2,2) Theorem]{Fiedler75}.
  \item [(iii)] For our purpose, we only need the upper bound estimate in (\ref{eq:edge}). The lower bound estimate in (\ref{eq:edge}) can provide an alternative proof of the estimate $\mathfrak{S}(f_k)\leq k+r-1$ for strong nodal domains in Theorem \ref{Thm:upper bound}.
\end{itemize}

\end{remark}

\begin{proof}
For the readers' convenience, we recall the proof here.
Let the rank of $B$ be $n-r$. By Sylvester's law of inertia, the quadratic form $B$ can be reformulated as
\begin{equation}
	B=\sum_{i=1}^{n-r}b_iY_i^2,
\end{equation}
where $Y_i=\sum_{j=1}^{n}m_{ij}x_j, i=1,\ldots, n$ are independent linear forms and \[b_1>0,\ldots,b_d>0,\,b_{d+1}<0,\ldots,b_{n-r}<0.\]
We argue by contradiction.

Suppose that $|E(H)|>d+\ell$. We consider the following two systems of linear equations
\begin{equation}\label{eq:5.5}
	Y_1=0,\ldots,Y_d=0,\,x_i-x_j=0, \,\text{for any}\,\{i,j\}\in E\setminus E(H),
\end{equation}
and
\begin{equation}\label{eq:5.6}
x_i-x_j=0, \,\text{for any}\,\{i,j\}\in E.
\end{equation}
Let $c$ be the number of connected components of $G$. By our assumption $|E(H)|>d+\ell$, we observe that
 \[\text{the rank of (\ref{eq:5.5})}\,\leq d+|E|-|E(H)|\leq d+n-c+\ell-|E(H)|<n-c\]
and, by Lemma \ref{Lemma:Fiedler}, the rank of (\ref{eq:5.6}) is $n-c$.
Hence, there exists a nonzero solution $(x_1^0,....,x_n^0)$ of (\ref{eq:5.5}) which fails (\ref{eq:5.6}). Then we derive $B(x_1^0,....,x_n^0)\leq 0$ from (\ref{eq:5.5}) and $B(x_1^0,....,x_n^0)>0$ from (\ref{eq:5.6}), which is a contradiction. This shows $|E(H)|\leq d+\ell$.

Suppose that $|E(T)|<d$. Let us consider the following two systems of linear equations
\begin{equation}\label{eq:5.3}
	x_i-x_j=0, \,\text{for any }\,\{i,j\}\in E(H),\,Y_{d+1}=0,\ldots,Y_{n-r}=0
\end{equation}
and
\begin{equation}\label{eq:5.4}
     Y_1=0,\ldots,Y_{n-r}=0.
\end{equation}
Now we compare the ranks of the two systems. By Lemma \ref{Lemma:Fiedler} and our assumption that $|E(T)|<d$, we estimate
\[\text{the rank of (\ref{eq:5.3})}\,\leq |E(T)|+n-r-d<n-r.\]
On the other hand, the rank of (\ref{eq:5.4}) is $n-r$. Therefore, there exists a nonzero solution $(x_1^0,.....,x_n^0)$ of (\ref{eq:5.3}) which fails (\ref{eq:5.4}). Then we derive $B(x_1^0,.....,x_n^0)\leq 0$ from (\ref{eq:5.3}) and $B(x_1^0,.....,x_n^0)>0$ from (\ref{eq:5.4}), which is a contradiction. This shows $|E(T)|\geq d$.
\end{proof}
Now, we are ready for the proof of Theorem \ref{Thm:lower bound}. First we consider the case that $f_k$ has no zeros.
\begin{lemma}\label{lemma:1}
Let $M$, $\Gamma=(G,\sigma)$, and $f_k$ be defined as in the Theorem \ref{Thm:lower bound}. If $f_k$ is non-zero at each vertex, then we have
\[\mathfrak{S}(f_k)\geq k+r-1-\ell+\ell_+.\]
\end{lemma}
\begin{proof}
	Let us denote the vertex set of $G$ by $V=\{1,2,\ldots,n\}$. Define $F$ to be the diagonal matrix with $F_{ii}=f_k(i)$ for any $i\in V$, and $B:=F(M-\lambda_{k}I)F$. By Corollary \ref{cor:XY}, we get for any function $g: V\to \mathbb{R}$
\begin{equation}
	\langle g,Bg \rangle=\sum_{\{i,j\}\in E}(-M_{ij})f_k(i)f_k(j)(g(i)-g(j))^2=\sum_{\{i,j\}\in E}a_{ij}(g(i)-g(j))^2
\end{equation}
where $a_{ij}:=(-M_{ij})f_k(i)f_k(j)$, which is nonzero if and only if $\{i,j\}\in E$.\\
Next we apply Lemma \ref{Lemma:5.3} to the quadratic form $B$ and the graph $G$. Let $H$ be the subgraph of $G$ defined as in Lemma \ref{Lemma:5.3} and $T$ be a spanning forest of $H$. We observe that the edge set of $H$ is exactly the set of edges in $G$ which are S-walks of $f_k$ on the signed graph $\Gamma=(G,\sigma)$.
Then we derive from Lemma \ref{lemma:strong} that
\begin{equation}\label{eq:5.8}
	\mathfrak{S}(f_k)=n-|E(T)|.
\end{equation}
Since $F$ is nonsingular, the positive index of inertia of $B$ satisfies
\[p_B=p_{F(M-\lambda_kI)F}=p_{(M-\lambda_kI)}=n-(k+r-1).\]
Therefore, we obtain by Lemma \ref{Lemma:5.3}
\begin{equation}
	n-(k+r-1)\leq |E(T)|\leq |E(H)|\leq n-(k+r-1)+\ell.
\end{equation}
Noticing that $|E(H)|=|E(T)|+\ell(H)=|E(T)|+\ell_+(G,\sigma,f_k)$, we derive
\begin{equation}\label{eq:5.9}
	n-(k+r-1)\leq |E(T)|\leq n-(k+r-1)+\ell-\ell_+.
\end{equation}
Inserting (\ref{eq:5.8}) into (\ref{eq:5.9}) yields
\begin{equation}\label{eq:5.10}
	k+r-1-\ell+\ell_+\leq \mathfrak{S}(f_k)\leq k+r-1.
\end{equation}
This proves the lemma.
\end{proof}
Next, we consider the case that every zero of $f_k$ is not in the Fiedler zero set.
\begin{lemma}\label{lemma:2}
	Let $M,\,\Gamma=(G,\sigma)$ and $f_k$ be defined as in Theorem \ref{Thm:lower bound}.
	If all zeros of $f_k$ lie in $\mathcal{F}^c=\mathcal{F}^c(G,f_k)$. Then we have
		\[\mathfrak{S}(f_k)\geq k+r-1-\ell'+\ell_+,\]
where $\ell'=\ell(G')$ and $G'$ is the induced subgraph of $G$ on the set of nonzeros.
\end{lemma}
\begin{remark}
Due to Proposition \ref{prop:treelike}, we have in the above that $\ell(G')=\ell(G)$, since all vertices in $\mathcal{F}^c$ are tree-like.
\end{remark}
\begin{proof}
	We denote the vertex set of $G$ by $V=\{1,2,\ldots,n\}$. Let us denote by $N$ the symmetric matrix obtained from $M$ by deleting all rows and columns with indices from $\mathcal{F}^c$. Then we claim that
\begin{equation}\label{eq:inertia}
p_{(M-\lambda_kI)}=p_{(N-\lambda_kI)}+|\mathcal{F}^c|,
\end{equation}
where $p_{(M-\lambda_kI)}$ and $p_{(N-\lambda_kI)}$ are the positive indices of inertia of $M-\lambda_kI$ and $N-\lambda_kI$, respectively. For ease of notation, we do not distinguish $I_n$ and $I_{n-|\mathcal{F}^c|}$.

We prove this claim by induction with respect to the number $|\mathcal{F}^c|$. When $|\mathcal{F}^c|=0$, the claim holds true. Next, we assume the claim is true when $|\mathcal{F}^c|=m-1$. We consider the case that $|\mathcal{F}^c|=m$.
	Without loss of generality, we assume $n\in \mathcal{F}^c$ and the matrix $M-\lambda_k I$ has the following form:
\[M-\lambda_k I_n=\left(
                      \begin{array}{cc}
                        M_0 & \eta \\
                        \eta^T & M_{nn}-\lambda_k \\
                      \end{array}
                    \right)
\]
with $\eta^T=(M_{n1},\ldots, M_{n(n-1)})$ and  \[M_0:=\left(
\begin{array}{cccc}
	M_1-\lambda_k I_{n_1} & 0 & \cdots &  0 \\
	0 &  M_2-\lambda_k I_{n_2} &\cdots &  0  \\
	\vdots & \vdots& \ddots & \vdots\\
	0 &  0 &\cdots &  M_h-\lambda_k I_{n_h} \\
\end{array}
\right),\]
where $h=d_n$ is the degree of $n$ and $M_i$ is an $n_i\times n_i$ symmetric matrix for each $i$.

Since $n\in \mathcal{F}^c$, there exists an index $j$ such that $M_{nj}f_k(j)\neq 0$. Without loss of generality, we assume $j\in \{1,\ldots,n_1\}$. We set
\[u:=(f_k(1),\ldots,f_k(n_1),0,\ldots,0)^T\in \mathbb{R}^{n-1}.\]
 Since the vertex $n$ is tree-like, we have $n\nsim \zeta$, for any $\zeta\in \{1,\ldots, n_1\}\setminus\{j\}$. Therefore, we derive \[\eta^Tu=\sum_{\zeta=1}^{n_1}M_{n\zeta}f_k(\zeta)=M_{nj}f_k(j)\neq 0.\]
 Moreover, we have $M_0u=0$.
  Then we can apply Lemma \ref{lemma:Fiedler2} to conclude that \[p_{(M-\lambda_kI)}=p_{M_0}+1.\] By our induction assumption, we have $p_{(M-\lambda_kI)}=p_{M_0}+1=p_{(N-\lambda_kI)}+|\mathcal{F}^c|$. That is, we prove the claim (\ref{eq:inertia}).

Let $\mu_1\leq\cdots\leq \mu_{n-|\mathcal{F}^c|}$ be the eigenvalues of $N$. We assume
\[\mu_{k'-1}<\lambda_k=\mu_{k'}=\cdots=\mu_{k'+r'-1}<\mu_{k'+r'}.\]
We observe that $p_{(M-\lambda_kI)}=n-(k+r-1)$, and $p_{(N-\lambda_kI)}=n-|\mathcal{F}^c|-(k'+r'-1)$. Then (\ref{eq:inertia}) implies
\begin{equation}\label{eq:strongInterlacing}
k+r=k'+r'.
\end{equation}
 Note that $G'$ is the induced subgraph of $N$. By definition, we have $\mathfrak{S}(f_k|_{G'})=\mathfrak{S}(f_k)$ since the set of nonzeros stays put. Applying Lemma \ref{lemma:1} and (\ref{eq:strongInterlacing}) leads to
\[k+r-1=k'+r'-1\geq \mathfrak{S}(f_k)=\mathfrak{S}(f_k|_{G'})\geq k'+r'-1-\ell'+\ell_+=k+r-1-\ell'+\ell_+.\]
This completes the proof.
\end{proof}

\begin{proof}[Proof of Theorem \ref{Thm:lower bound}]
 Restrict the function $f_k$  to the induced subgraph $\widetilde{G}$ of $G$ on $V\setminus \mathcal{F}$. Then $f_k$ is still an eigenfunction of $M|_{\widetilde{G}}$ restricting to $\widetilde{G}$ corresponding to the eigenvalue $\lambda_{k}$. We denote by $\mu_1\leq\cdots\leq \mu_{n-|\mathcal{F}|}$ the eigenvalues of $M|_{\widetilde{G}}$. We assume
\[\mu_{\tilde{k}-1}<\lambda_k=\mu_{\tilde{k}}=\cdots=\mu_{\tilde{k}+\tilde{r}-1}<\mu_{\tilde{k}+\tilde{r}}.\]
Observing that all zeros of $f_k|_{\widetilde{G}}$ lie in $\mathcal{F}^c(\widetilde{G},f_k|_{\widetilde{G}})$, we obtain by Lemma \ref{lemma:2}
\begin{equation}\label{eq:5.11}
	\mathfrak{S}(f_k)=\mathfrak{S}(f_k|_{\tilde{G}})\geq \tilde{k}+\tilde{r}-1-\ell'+\ell_+,
\end{equation}
where $\ell'=\ell(G')$. Recall that $G'$ is the induced subgraph of $G$ on the set of nonzeros.

Applying the interlacing result Lemma \ref{thm:interlacing}, we have
\[\lambda_{\tilde{k}+\tilde{r}+|\mathcal{F}|}\geq \mu_{\tilde{k}+\tilde{r}}>\mu_{\tilde{k}+\tilde{r}-1}=\lambda_{k}=\lambda_{k+r-1}.\]
This implies that
\begin{equation}\label{eq:Interlacing}
\tilde{k}+\tilde{r}+|\mathcal{F}|\geq k+r.
\end{equation}
Inserting (\ref{eq:Interlacing}) into (\ref{eq:5.11}) yields
\[\mathfrak{S}(f_k)\geq k+r-|\mathcal{F}|-1-\ell'+\ell_+.\]
This completes the proof.
\end{proof}

%
%
\begin{remark}
In fact, we have an alternative proof of the estimate $\mathfrak{S}(f_k)\leq k+r-1$ in Theorem \ref{Thm:upper bound} by combining (\ref{eq:5.10}) and the interlacing theorem. Let $G'$ be the induced subgraph of $G$ on the set of nonzeros of $f_k$. Let $\mu_1\leq\cdots\leq \mu_{n-z}$ be the eigenvalues of $M|_{G'}$, where $z=|\mathcal{F}|+|\mathcal{F}^c|$ is the total number of zeros of $f_k$. We can assume $\lambda_{k-1}<\lambda_k$, $\lambda_k=\mu_{k'}$ and $\mu_{k'-1}<\mu_{k'}=\cdots=\mu_{k'+r'-1}<\mu_{k'+r'}$.  By Lemma \ref{thm:interlacing}, we derive
\[\lambda_{k'+r'-1}\leq \mu_{k'+r'-1}=\lambda_{k+r-1}.\]
This leads to $k'+r'\leq k+r$.
Inserting it into (\ref{eq:5.10}) leads to
$\mathfrak{S}(f_k)\leq k+r-1.$
\end{remark}

Next, we discuss another application of Lemma \ref{Lemma:5.3}.
\begin{theorem}\label{thm:xuyau}
Let $M=(M_{ij})$ be an $n\times n$ symmetric matrix. Let $G=(V,E)$ be the graph with $V=\{1,2,\ldots,n\}$ and $E=\{\{i,j\}: M_{ij}\neq 0 \,\text{and}\,i\neq j\}$.
Then the multiplicity $r$ of an eigenvalue $\lambda$ of $M$ with an eigenfunction $f$ non-zero at each vertex satisfies
\[c\leq r\leq c+\ell,\]
where $c$ is the number of connected components of $G$ and $\ell=\ell(G)$.
\end{theorem}
\begin{remark}
Theorem \ref{thm:xuyau} is due to Xu and Yau \cite[Corollary 2.5]{XY12}, where they state it for generalized Lapalcians.
\end{remark}
\begin{proof}
	Since $f$ has no zeros, the restriction of $f$ to any connected component is also an eigenfunction of the same eigenvalue. Therefore, we have $r\geq c$.
	
Consider the quadratic form
\[B=\sum_{\{i,j\}\in E}f(i)(-M_{ij})f(j)(x_i-x_j)^2.\]
We have a partition of the edge set $E=E_1\cup E_2$, where \[E_1=\{\{i,j\}\in E ,\,-f(i)M_{ij}f(j)>0\},\,E_2=\{\{i,j\}\in E,\, -f(i)M_{ij}f(j)<0\}.\]
Applying Lemma \ref{Lemma:5.3} to the quadratic forms $B$ and $-B$, respectively, we obtain
\[|E_1|\leq p+\ell,\,\,\text{and}\,\,|E_2|\leq n-p-r+\ell,\]
where $p$ stands for the positive index of inertia of $B$. Then we derive
	\[|E_1|+|E_2|\leq n+2\ell -r.\]
	
	Since $|E_1|+|E_2|=|E|=n-c+\ell$, we have
	$r\leq c+\ell$.
 \end{proof}

When the induced signed graph has a large number of leaves, i.e., vertices with degree $1$, we have the following non-trivial lower bound estimate via a duality argument.
 \begin{theorem}\label{Thm:lower bound 2}
 Let $M$ be an $n\times n$ symmetric matrix and $\Gamma=(G,\sigma)$ be the induced signed graph. Let $\lambda_{k}$ be the $k$-th eigenvalue of $M$ with an eigenfunction $f_k$ such that $\lambda_{k-1}<\lambda_k$.
Then we have
\[\mathfrak{S}(f_k)\geq k+v_l-1-n-z_l+z_r,\] where $z_l=|\{x: d_x=1, f_k(x)=0\}|$, $z_r=|\{x: d_x=1, f_k(x)\neq 0, f_k(x')=0,\,\text{for}\,\,x'\sim x\}|$, and $v_l$ is the number of leaves.
 \end{theorem}
\begin{proof}
Let $r$ be the multiplicity of $\lambda_k$. Recall that $\overline{\mathfrak{S}}(f_k)$ is the number of strong nodal domains of $f_k$ on $-\Gamma$. Theorem \ref{Thm:upper bound} implies $\overline{\mathfrak{S}}(f_k)\leq n-(k+r-1)+1+r-1=n-k+1$.

Let $x$ be a leaf and $f_k(x)\neq 0$. We denote by $x'$ the only vertex such that $\{x,x'\}\in E$. If $f_k(x')=0$, then the subgraph induced by $\{x\}$ is a strong nodal domain of $f_k$ on both $\Gamma$ and its negation $-\Gamma$. If $f_k(x')\neq 0$, then either $f_k(x)M_{xx'}f_k(x')>0$ or $f_k(x)M_{xx'}f_k(x')<0$. This means that the subgraph induced by $\{x\}$ is a strong nodal domain of $f_k$ on either $\Gamma$ or $-\Gamma$. Therefore, we derive
\begin{equation}\label{eq:strongDual}
\mathfrak{S}(f_k)+\overline{\mathfrak{S}}(f_k)\geq (v_l-z_l-z_r)+2z_r=v_l-z_l+z_r.
\end{equation}
Then we obtain
\[\mathfrak{S}(f_k)\geq v_l-z_l+z_r-\overline{\mathfrak{S}}(f_k)\geq v_l-z_l+z_r-n+k-1.\]
\end{proof}
To conclude this section, we compare the two lower bound estimates of Theorems \ref{Thm:lower bound} and \ref{Thm:lower bound 2} in the following examples.
\begin{example}\label{example:win}
Let $G=(V,E)$ be a graph obtained by adding $7$ leaves to one vertex of the complete graph $K_7$. We consider the symmetric matrix $M=-A$ where $A$ is the adjacent matrix of $G$. Then the induced singed graph of $M$ is $\Gamma=(G,\sigma)$, where $\sigma\equiv +1$.
Let us denote $V=\{1,2,\ldots,14\}$. Let $\{1,2,\ldots,7\}$ be the vertices of the graph $K_7$. Assume that the 7 leaves $\{8,9,\ldots,14\}$ are added to the vertex $1$.
The largest eigenvalue $\lambda_{14}\approx3.05$ of $M$ is simple and its eigenfunction is
\[f_{14}\approx(-3.05,0.38,0.38,0.38,0.38,0.38,0.38,1,1,1,1,1,1,1)^T.\]
It is direct to figure out $\mathfrak{S}(f_{14})=9$, $\ell=|E|-|V|+1=15$, $\ell_+=10$ and $v_{l}=7$.  Theorem \ref{Thm:lower bound} tells $\mathfrak{S}(f_{14})\geq 9$, which is sharp, while Theorem \ref{Thm:lower bound 2} tells $\mathfrak{S}(f_{14})\geq 6$.
\end{example}
In the next example, Theorem \ref{Thm:lower bound 2} provides a better estimate than Theorem \ref{Thm:lower bound}.
\begin{example}
	We consider a complete graph $K_8$ and denote its vertex set by $\{1,\ldots ,8\}$. Let $G=(V,E)$ be a graph obtained by adding $i$ leaves to each vertex $i\in \{1,\ldots,8\}$.   For each $i$, we label the $i$ leaves adjacent to it by $\{8+\frac{i(i-1)}{2}+k: k=1,\ldots,i\}$. Then the vertex set $V=\{1,2,\ldots,44\}$. Consider $M=-A$ where $A$ is the adjacent matrix of $G$. Hence, the induced singed graph of $M$ is $\Gamma=(G,\sigma)$, where $\sigma\equiv +1$. The $41$-st eigenvalue $\lambda_{41}\approx 2.69$ of $M$ is simple and its eigenfunction is
	\begin{align*}
		f_{41} \approx (0.07, & 0.10,  0.16, 0.47, -0.55, -0.17, -0.10, -0.07, -0.02, -0.04, -0.04,\\
		 -0.&06, -0.06,-0.06,-0.17,-0.17,-0.17,-0.17,0.20,0.20,0.20,0.20,\\
		&0.20,0.06,0.06,0.06,0.06,0.06,0.06,0.04,0.04,0.04,0.04,\\
		&0.04,0.04,0.04,0.03,0.03,0.03,0.03,0.03,0.03,0.03,0.03)^T.
	\end{align*}
Notice that the values of $f_{41}$ has different signs at a leaf and the vertex adjacent to it.
It is direct to compute $\mathfrak{S}(f_{41})=38,\,\ell=21,\,\ell_+=6$, and $v_l=36$. Theorem \ref{Thm:lower bound} tells that $\mathfrak{S}(f_{41})\geq41-21+6=26$ while Theorem \ref{Thm:lower bound 2} tells that $\mathfrak{S}(f_k)\geq 41+36-1-44=32$.
\end{example}
\appendix

\section*{Acknowledgement}
We are very grateful to Dong Zhang for discussions on discrete nodal domain theorems of signless Laplacians. We thank Ali Mohammadian for bringing his interesting work \cite{Mohammadian16} to our attention after the submission of our first arXiv version.
This work is supported by the National Key R and D Program of China 2020YFA0713100 and the National Natural Science Foundation of China (No. 12031017).

\end{document}